\documentclass[reqno,11pt]{amsart}
\usepackage{amsmath,amsfonts,amscd,amssymb,latexsym}

\usepackage{hyperref}
\usepackage{soul}
\usepackage{mathrsfs}  

\usepackage[nobysame]{amsrefs}
\usepackage{comment}
\usepackage{soul}
\usepackage{tikz}
\usepackage{marginnote}
\usepackage{anysize}
\usepackage{bbm}
\usepackage{cancel}
\usepackage[normalem]{ulem}

\numberwithin{equation}{section}

\theoremstyle{plain}
\newtheorem{theorem}[subsection]{Theorem}
\newtheorem{lemma}[subsection]{Lemma}
\newtheorem{corollary}[subsection]{Corollary}
\newtheorem{proposition}[subsection]{Proposition}
\newtheorem{conjecture}[subsection]{Conjecture}

\theoremstyle{definition}
\newtheorem{definition}[subsection]{Definition}
\newtheorem{example}[subsection]{Example}

\theoremstyle{remark}
\newtheorem{remark}[subsection]{Remark}

\newcommand{\ind}{\operatorname{ind}}
\newcommand{\id}{\operatorname{id}}
\newcommand{\dom}{\operatorname{dom}}
\newcommand{\loc}{\operatorname{loc}}
\newcommand{\cc}{{cc}}
\newcommand{\dm}{\partial M}
\newcommand{\CC}{\mathbb{C}}
\newcommand{\RR}{\mathbb{R}}
\newcommand{\cRR}{\mathscr{R}}
\newcommand{\ZZ}{\mathbb{Z}}

\newcommand{\AAA}{\mathcal{A}}

\newcommand{\DD}{\mathcal{D}}
\newcommand{\EE}{\mathscr{E}}

\newcommand{\cHH}{\check{H}}
\newcommand{\hHH}{\hat{H}}

\newcommand{\bfu}{\mathbf{u}}
\newcommand{\bfv}{\mathbf{v}}
\newcommand{\upper}{\uppercase\expandafter}

\newcommand{\n}{\nabla}
\newcommand{\p}{\partial}
\newcommand{\pM}{{\p M}}
\newcommand{\Hloc}{H_{\loc}}

\newcommand{\supp}{\operatorname{supp}}
\newcommand{\End}{\operatorname{End}}
\newcommand{\ad}{{\rm ad}}
\newcommand{\Tr}{\operatorname{Tr}}
\newcommand{\oB}{\bar{B}}
\newcommand{\spf}{\operatorname{sf}}
\newcommand{\PP}{\mathbb{P}}
\renewcommand{\AA}{\mathbb{A}}
\newcommand{\IM}{\operatorname{im}}

\newcommand{\ou}{\bar{u}}
\newcommand{\tilM}{\tilde M}
\newcommand{\tilE}{\tilde E}
\newcommand{\tilv}{\tilde v}
\newcommand{\tilu}{\tilde u}
\newcommand{\tilw}{\tilde w}

\begin{document}

\title[APS index with non-compact boundary]{The Atiyah--Patodi--Singer index on manifolds with non-compact boundary}


\author{Maxim Braverman${}^\dag$}
\address{Department of Mathematics,
Northeastern University,
Boston, MA 02115,
USA}

\email{maximbraverman@neu.edu}
\urladdr{www.math.neu.edu/~braverman/}

\author{Pengshuai Shi}
\address{Beijing International Center for Mathematical Research (BICMR), 
Peking University, 
Beijing 100871, China}

\email{pengshuai.shi@gmail.com}

\thanks{${}^\dag$†Partially supported by the Simons Foundation collaboration grant \#G00005104.}


\begin{abstract}
We study the index of  the APS boundary value problem for a strongly Callias-type operator $\DD$ on a complete Riemannian manifold $M$. We show that this index is equal to an index on a simpler manifold whose boundary is a disjoint union of two complete manifolds $N_0$ and $N_1$.  If the dimension of $M$ is odd we show that the latter index depends only on the  restrictions $\AAA_0$ and $\AAA_1$ of $\DD$ to $N_0$ and $N_1$ and thus is an invariant of the boundary. We use this invariant to define the relative $\eta$-invariant $\eta(\AAA_1,\AAA_0)$. We show that even though in our situation the $\eta$-invariants of $\AAA_1$ and $\AAA_0$ are not defined, the relative $\eta$-invariant behaves as if it were the difference $\eta(\AAA_1)-\eta(\AAA_0)$. 
\end{abstract}

\maketitle

\setcounter{tocdepth}{1}
\tableofcontents

\section{Introduction}\label{S:intro}

A Callias-type operator on a complete Riemannian manifold is an operator of the form $\DD= D+i\Phi$ where $D$ is a Dirac operator and $\Phi$ is a self-adjoint potential which commutes with the Clifford multiplication and satisfies certain growth conditions at infinity, so that $\DD$ is Fredholm. By the celebrated Callias-type index theorem, proven in different generalities in  \cites{Callias78,BottSeeley78,BruningMoscovici,Anghel93Callias,Bunke95}, the index of a Callias-type operator on a complete {\em odd-dimensional} manifold is equal to the index of a certain operator induced by $\Phi$ on a compact hypersurface. Several generalizations and applications of the Callias-type index theorem were obtained recently in \cites{Kottke11,CarvalhoNistor14,Wimmer14,Kottke15,BrShi16,BrCecchini17}.

B\"ar and Ballmann, \cites{BaerBallmann12,BaerBallmann13}, showed that an elliptic  boundary value problem for  a Callias-type   operator  on a complete manifold with {\em compact} boundary is Fredholm and studied its index.  P.~Shi, \cite{Shi17}, proved a version of the Callias-type index theorem for the Atiyah--Patodi--Singer (APS)  boundary value problem for Callias-type operators on a complete manifold with compact boundary. 

The study of Callias-type operators on manifolds  with non-compact boundary was initiated by Fox and Haskell \cites{FoxHaskell03,FoxHaskell05}. Under rather strong conditions on the manifold and the operator $\DD$ they showed that the heat kernel of $\DD^*\DD$ has a nice asymptotic expansion and proved a version of the Atiyah--Patodi--Singer index theorem in this situation. 

The purpose of this paper is to study the index of the APS boundary value problem on an arbitrary complete {\em odd-dimensional} manifold $M$ with non-compact boundary without introducing any extra assumptions on manifold (in particular, we do not assume that our manifold is of bounded geometry). Note that as for the Callias-type theorem for manifolds without boundary, we consider odd-dimensional case so that a compact hypersurface in $M$ has even dimension. 

We now briefly describe our main results.

\subsection{Index of a boundary value problem for manifolds with non-compact boundary}\label{SS:Iindex}
Let $M$ be a complete Riemannian manifold with non-compact boundary $\pM$ and let $\DD=D+i\Phi$ be a Callias-type operator on $M$. We impose slightly stronger conditions on the growth of the potential $\Phi$ and call the operators satisfying these conditions {\em strongly Callias-type}. On manifolds without boundary these conditions guarantee that $\DD$ has a discrete spectrum. 

The restriction $\AAA$ of a strongly Callias-type  operator to the boundary is a self-adjoint strongly Callias-type operator on $\pM$ and, hence, has a discrete spectrum. In Sections~\ref{S:strCallias&maxdom} we use the eigensections of $\AAA$ to define a scale of Sobolev spaces $H^s_{\AAA}(\pM,E_\pM)$ on $\pM$ (this scale does depend on the operator $\AAA$). In Section~\ref{S:bvp} we use this scale to define elliptic boundary conditions for $\DD$. This definition is completely analogous to the classical construction \cite{BaerBallmann12}, but depends more heavily on $\AAA$, since the Sobolev spaces depend on $\AAA$.
These parts are somewhat parallel to the corresponding sections in \cite{BaerBallmann12}, but we feel it is necessary to set them up here. This is not only for the completeness of the paper but also because of the fact that some extra care should be taken to obtain the results due to the non-compactness.

Our first main result, Theorem~\ref{T:Fred}, is that a strongly Callias-type operator with elliptic boundary condition is Fredholm. This generalizes a theorem of B\"ar and Ballmann to manifolds with non-compact boundary. We also extend some standard properties of the index of boundary value problems on compact manifolds to our non-compact setting. In particular, we establish  a Splitting Theorem~\ref{T:splitting}: if $M=M_1\cup_N M_2$ where $N$ is a not necessarily compact hypersurface, then index on $M$ is equal to the sum of the indexes of a boundary value problem on $M_1$ and the dual boundary value problem on $M_2$.

\subsection{An almost compact essential support}\label{SS:Iesssupport}
In the theory of Callias-type operators on a manifold without boundary the crucial notion is that of the {\em essential support} --  a compact set $K\subset M$ such that the restriction of $\DD^*\DD$ to $M\backslash K$ is strictly positive. For manifolds with boundary we want an analogous subset, but the one which has the same boundary as $M$ (so that we can keep the boundary conditions). Such a set is necessarily non-compact. In Section~\ref{S:esscylindrical}, we introduce a class of non-compact manifolds, called {\em essentially cylindrical} manifolds, which replaces the class of compact manifolds in our study. An essentially cylindrical manifold is a manifold which outside of a compact set looks like a cylinder $[0,\varepsilon]\times N'$, where $N'$ is a non-compact manifold. The boundary of an essentially cylindrical manifold is a disjoint union of
two complete manifolds $N_0$ and $N_1$ which are isometric outside of a compact set. 

We say that an essentially cylindrical manifold $M_1$, which contains $\pM$,  is an {\em almost compact essential support of\/ $\DD$}\/ if the restriction of $\DD^*\DD$ to $M\backslash M_1$ is strictly positive and the restriction of $\DD$ to the cylinder $[0,\varepsilon]\times N'$ is a product, cf. Definition~\ref{D:almost compact support}. We show that every strongly Callias-type operator on $M$ which is a product near $\pM$ has an almost compact essential support. 

The main result of Section~\ref{S:esscylindrical} is that {\em the index of the APS boundary value problem for a strongly Callias-type operator $\DD$ on a complete {\em odd-dimensional} manifold $M$ is equal to the index of the APS boundary value problem of the restriction of $\DD$ to its almost compact essential support $M_1$}, cf. Theorem~\ref{T:indM=indM1}.

\subsection{Index on an essentially cylindrical manifold}\label{SS:Iesscyl}
In the previous section we reduced the study of the index of the APS boundary value problem on an arbitrary complete odd-dimensional manifold to index on an essentially cylindrical manifold. A systematic study of the latter is done in Section~\ref{S:esscyl}.

Let $M$ be an essentially cylindrical manifold and let $\DD$ be a strongly Callias-type operator on $M$, whose restriction to the cylinder $[0,\varepsilon]\times N'$ is a product. Suppose $\pM= N_0\sqcup N_1$ and denote the restrictions of $\DD$ to $N_0$ and $N_1$ by $\AAA_0$ and $-\AAA_1$ respectively (the sign convention means that we think  of  $N_0$ as the ``left boundary" and of $N_1$ as the ``right boundary" of $M$). 

Our main result here is that {\em the index of the APS boundary value problem for $\DD$ depends only on the operators $\AAA_0$ and $\AAA_1$} and not on the interior of the manifold $M$ and the restriction of\/ $\DD$ to the interior  of $M$, cf. Theorem~\ref{T:indep of D}. The odd-dimensionality  of $M$ is essential, since the proof uses the Callias-type index theorem on complete manifolds without boundary.

\subsection{The relative $\eta$-invariant}\label{SS:Ireleta}
Suppose now that $\AAA_0$ and $\AAA_1$ are self-adjoint strongly Callias-type operators on complete {\em even-dimensional} manifolds $N_0$ and $N_1$ respectively. An {\em almost compact cobordism} between $\AAA_0$ and $\AAA_1$ is an essentially cylindrical manifold $M$ with $\pM=N_0\sqcup N_1$ and a strongly Callias-type operator $\DD$ on $M$, whose restriction to the cylindrical part of $M$ is a product and such that the restrictions of $\DD$ to $N_0$ and $N_1$ are equal to $\AAA_0$ and $-\AAA_1$ respectively. We say that $\AAA_0$ and $\AAA_1$ are {\em cobordant}\/ if there exists an almost compact cobordism between them. Note that this means, in particular, that $\AAA_0$ and $\AAA_1$ are equal outside of a compact set. 

Let  $\DD$ be an almost compact cobordism between $\AAA_0$ and $\AAA_1$. Let $B_0$ and $B_1$ be the APS boundary conditions for $\DD$ at $N_0$ an $N_1$ respectively. Let $\ind \DD_{B_0\oplus B_1}$ denote the index of the APS boundary value problem for $\DD$.  We define the {\em relative $\eta$-invariant} by the formula
\begin{equation}\label{E:Ireleta}\notag
	\eta(\AAA_1,\AAA_0) \ = \  2\,\ind \DD_{B_0\oplus B_1}
	 \ + \ \dim\ker \AAA_0\ + \ \dim\ker \AAA_1.
\end{equation}
It follows from the result of the previous section, that $\eta(\AAA_1,\AAA_0)$ is independent of the choice of an almost compact cobordism. 

Notice the ``shift of dimension" of the manifold compared to the theory of $\eta$-invariants on compact manifolds. This is similar to the ``shift of dimension" in the Callias-type index theorem: on compact manifolds the index of elliptic operators is interesting for even-dimensional manifolds, while for Callias-type operators it is interesting for odd-dimensional manifolds. Similarly, the theory of $\eta$-invariants on compact manifolds is more interesting on odd-dimensional manifolds, while our relative $\eta$-invariant is defined on even-dimensional non-compact manifolds. 

If $M$ is a compact odd-dimensional manifold, then the Atiayh-Patodi-Singer index theorem \cite{APS1} implies that $\eta(\AAA_1,\AAA_0)= \eta(\AAA_1)-\eta(\AAA_0)$ (recall that since the dimension of $M$ is odd, the integral term in the index formula vanishes). In general, for non-compact manifolds, the individual $\eta$-invariants $\eta(\AAA_1)$ and $\eta(\AAA_0)$ might not be defined. However, we show that  $\eta(\AAA_1,\AAA_0)$ in many respects behaves like it was a difference of two individual $\eta$-invariants. In particular, we show, cf. Propositions~\ref{P:antisymmetry eta}-\ref{P:cocycle}, that 
\[
	\eta(\AAA_1,\AAA_0)\ = \ -\,\eta(\AAA_0,\AAA_1), 
	\qquad
	\eta(\AAA_2,\AAA_0)\ = \ \eta(\AAA_2,\AAA_1)\ + \ \eta(\AAA_1,\AAA_0).
\]

In \cite{FoxHaskell05} Fox and Haskell studied the index of a boundary value problem on manifolds of bounded geometry. They showed that under rather strong conditions on both $M$ and $\DD$ (satisfied for natural operators on manifolds with conical or cylindrical ends),  the heat kernel $e^{-t(\DD_B)^*\DD_B}$ is of trace class and its trace has an  asymptotic expansion similar to the one on compact manifolds. In this case the $\eta$-invariant can be defined by the usual analytic continuation of the $\eta$-function. We prove, cf. Proposition~\ref{P:FoxHaskell}, that under the assumptions of Fox and Haskell, our relative $\eta$-invariant satisfies $\eta(\AAA_1,\AAA_0)= \eta(\AAA_1)-\eta(\AAA_0)$.

More generally, it is often the case that the individual $\eta$-functions $\eta(s;\AAA_1)$ and $\eta(s;\AAA_0)$ are not defined,  but their difference $\eta(s;\AAA_1)-\eta(s;\AAA_0)$ is defined and regular at 0. Bunke, \cite{Bunke92}, studied the case of the undeformed Dirac operator $A$  and gave geometric conditions under which $\Tr(A_1e^{-tA_1^2}- A_0e^{-tA_0^2})$ has a nice asymptotic expansion. In this  case he defined the relative $\eta$-function using the usual formula, and showed that it has a meromorphic extension to the whole plane, which is regular at 0. He defined the  relative $\eta$-invariant as the value of the relative $\eta$-function at 0. There are also many examples of strongly Callias-type operators for which the difference of heat kernels $\AAA_1e^{-t\AAA_1^2}- \AAA_0e^{-t\AAA_0^2}$ is of trace class and the relative $\eta$-function can be defined by the formula similar to \cite{Bunke92}. We conjecture that in this situation our relative $\eta$-invariant $\eta(\AAA_1,\AAA_0)$ is equal to the value of the relative $\eta$-function at 0.

\subsection{The spectral flow}\label{SS:Isp flow}

Atiyah, Patodi and Singer, \cite{APS3}, introduced a notion of spectral flow $\spf(\AA)$ of a smooth family $\AA:= \{\AAA^s\}_{0\le s\le 1}$ of self-adjoint differential operators on closed manifolds as the integer that counts the net number of eigenvalues that change sign when $s$ changes from 0 to 1.  They showed that the spectral flow  computes the variation of the $\eta$-invariant $\eta(\AAA^1)-\eta(\AAA^0)$. 

In Section~\ref{S:sp flow} we consider a family of self-adjoint strongly Callias-type operators $\AA= \{\AAA^s\}_{0\le s\le 1}$ on a complete {\em even-dimensional} Riemannian manifold. We assume that there is a compact set $K\subset M$ such that the restriction of $\AAA^s$ to $M\backslash K$ is independent of $s$. Then all $\AAA^s$ are cobordant in the sense of Section~\ref{SS:Ireleta}. Since the spectrum of $\AAA^s$ is discrete for all $s$, the spectral flow can be defined in more or less usual way. We show, Theorem~\ref{T:sp flow}, that 
\[
	\eta(\AAA^1,\AAA^0)\ = \ 2\,\spf(\AA).
\]
Moreover, if  $\AAA_0$ is another self-adjoint strongly Callias-type operator which is cobordant to $\AAA^0$ (and, hence, to all $\AAA^s$), then
\[	
	\eta(\AAA^1,\AAA_0) \ - \ \eta(\AAA^0,\AAA_0) \ = \ 2\,\spf(\AA).
\]

\subsection{Further developments}\label{SS:further}
After the first version of this paper was released several applications and improvements of these ideas have been developed. In \cite{BrShi17b} the case of even-dimensional manifolds was considered. This case is quite different from the odd-dimensional case. The odd-dimensional case should be viewed as a generalization of the Callias index theorem to manifolds with boundary, while the even-dimensional case is rather a generalization of the Atiyah--Patodi--Singer theorem to non-compact setting. 

In \cite{Shi18} the description of  the Cauchy data spaces for non-compact boundary was given. This result was used in  \cite{BrShi19local} to study the local boundary value problem for Callias-type operators on manifolds with non-compact boundary. This generalizes an index theorem of Freed, \cite{Freed98}, and gives a new insight on the Horava-Witten anomaly \cite{HoravaWitten96}. 

In \cite{Br19Lorentz} an APS-index formula on globally hyperbolic Lorentzian manifolds with non-compact Cauchy hypersurface was obtained, generalizing the theory of B\"ar and Strohmaier \cites{BarStrohmaier15,BarStrohmaier16}.

All these results  relay on the techniques developed in the current paper.

\section{Operators on a manifold with non-compact boundary}\label{S:mfldwbd}

In this section we  discuss different domains for operators on manifolds with boundary.

\subsection{Setting and notations}\label{SS:setting}

Let $M$ be a complete Riemannian manifold with  (possibly noncompact) boundary $\pM$. 
We denote the Riemannian metric on $M$ by $g^M$ and its restriction to the boundary by $g^\pM$.  Then $(\pM,g^\pM)$ is also a complete Riemannian manifold.  We denote by $dV$ the volume form on $M$ and by $dS$ the volume form on $\dm$. The interior of $M$ is denoted by $\mathring{M}$. For a vector bundle $E$ over $M$, $C^\infty(M,E)$ is the space of smooth sections of $E$, $C_c^\infty(M,E)$ is the space of smooth sections of $E$ with compact support, and $C_\cc^\infty(M,E)$ is the space of smooth sections of $E$ with compact support in $\mathring{M}$. Note that
\[
C_\cc^\infty(M,E)\subset C_c^\infty(M,E)\subset C^\infty(M,E).
\]
We denote by $L^2(M,E)$ the Hilbert space of square-integrable sections of $E$, which is the completion of $C_c^\infty(M,E)$ with respect to the norm induced by the $L^2$-inner product
\[
(u_1,u_2)_{L^2(M,E)}\;:=\;\int_M\,\langle u_1,u_2\rangle\, dV,
\]
where $\langle\cdot,\cdot\rangle$ denotes the fiberwise inner product. Similarly, we have spaces $C^\infty(\dm,E_{\dm})$, $C_c^\infty(\dm,E_{\dm})$ and $L^2(\dm,E_{\dm})$ on the boundary $\dm$, where $E_{\dm}$ denotes the restriction of the bundle $E$ to $\dm$. If $u\in C^\infty(M,E)$, we denote by $u_{\dm}\in C^\infty(\dm,E_{\dm})$ the restriction of $u$ to $\dm$. For general sections on the boundary $\dm$, we use bold letters $\bfu,\bfv,\cdots$ to denote them.

Let $E,F$ be two Hermitian vector bundles over $M$ and $D:C^\infty_c(M,E)\to C^\infty_c(M,F)$ be a first-order differential operator. The \emph{formal adjoint} of $D$, denoted by $D^*$, is defined by
\begin{equation}\label{E:formal adjoint}
		\int_M\langle Du,v\rangle dV\;=\;\int_M\langle\, u,D^*v\rangle\, dV,
\end{equation}
for all $u,v\in C_\cc^\infty(M,E)$. If $E=F$ and $D=D^*$, then $D$ is called \emph{formally self-adjoint}.

\subsection{Minimal and maximal extensions}\label{SS:minmaxext}

We set $D_\cc:=D|_{C_\cc^\infty(M,E)}$ and view it as an unbounded operator from $L^2(M,E)$ to $L^2(M,F)$. The \emph{minimal extension} $D_{\min}$ of $D$ is the operator whose graph is the closure of that of $D_\cc$. The \emph{maximal extension} $D_{\max}$ of $D$ is defined to be $D_{\max}=\big((D^*)_{cc}\big)^\ad$, where the superscript ``$\ad$'' denotes the adjoint of the operator in the sense of functional analysis. Both $D_{\min}$ and $D_{\max}$ are closed operators. Their domains, $\dom D_{\min}$ and $\dom D_{\max}$, become Hilbert spaces equipped with the \emph{graph norm} $\|\cdot\|_D$, which is the norm associated with the inner product
\[
	(u_1,u_2)_D\;:=\;
	\int_M\,\big(\langle u_1,u_2\rangle+\langle Du_1,Du_2\rangle\big)\,dV.
\]
It's easy to see from the following Green's formula that $C_c^\infty(M,E)\subset\dom D_{\max}$.

\subsection{Green's formula}\label{SS:Greensfor}

Let $\tau\in TM|_{\dm}$ be the unit inward normal vector field along $\dm$. Using the Riemannian metric, $\tau$ can be identified with its associated one-form. We have the following formula (cf. \cite[Proposition 3.4]{BoosWoj93book}).

\begin{proposition}[Green's formula] \label{P:Greensfor}
Let $D$ be as above. Then for all $u\in C_c^\infty(M,E)$ and $v\in C_c^\infty(M,F)$,
\begin{equation}\label{E:Greensfor}
	\int_M\,\langle Du,v\rangle\, dV\;=\;
	\int_M\,\langle u,D^*v\rangle\, dV\,-\,
	\int_{\dm}\,\langle\sigma_D(\tau)u_{\dm},v_{\dm}\rangle\, dS,
\end{equation}
where $\sigma_D$ denotes the principal symbol of the operator $D$.
\end{proposition}

\begin{remark}\label{R:Greensfor}
A more general version of formula \eqref{E:Greensfor} will be presented in Theorem \ref{T:maxdom} below.
\end{remark}

\subsection{Sobolev spaces}\label{SS:sobsp}

Let $\nabla^E$ be a Hermitian connection on $E$. For any $u\in C^\infty(M,E)$, the covariant derivative $\nabla^E u\in C^\infty(M,T^*M\otimes E)$. Applying the covariant derivative multiple times we get $(\nabla^E)^k\in C^\infty(M,T^*M^{\otimes k}\otimes E)$ for $k\in\ZZ_+$. We define \emph{$k^{th}$ Sobolev space} by 
\[
	H^k(M,E)\;:=\;
	\big\{\,u\in L^2(M,E):\,
	  (\nabla^E)^ju\in L^2(M,T^*M^{\otimes j}\otimes E)\text{ for all }j=1,\ldots,k\,\big\},
\]
where the covariant derivatives are understood in distributional sense. It is a Hilbert space with $H^k$-norm
\[
	\|u\|_{H^k(M,E)}^2\;:=\;
	\|u\|_{L^2(M,E)}^2+\|\nabla^E u\|_{L^2(M,T^*M\otimes E)}^2+\cdots+
	    \|(\nabla^E)^ku\|_{L^2(M,T^*M^{\otimes k}\otimes E)}^2.
\]
Note that when $M$ is compact, $H^k(M,E)$ does not depend on the choices of $\nabla^E$ and the Riemannian metric, but when $M$ is noncompact, it does.

We say $u\in L_{\loc}^2(M,E)$ if the restrictions of $u$ to compact subsets of $M$ have finite $L^2$-norms. For $k\in\ZZ_+$, we say $u\in H_{\loc}^k(M,E)$, the \emph{$k^{th}$ local Sobolev space}, if $u,\nabla^Eu,(\nabla^E)^2u,\dots,(\nabla^E)^ku$ all lie in $L_{\rm loc}^2$. This Sobolev space is independent of the preceding choices.
 
Similarly, we fix a Hermitian connection on $F$ and define the spaces $L^2(M,F)$, $L^2_{\loc}(M,F)$, $H^k(M,F)$, and $H^k_{\loc}(M,F)$. Again, definitions of these spaces apply without change to $\dm$.

\subsection{Completeness}\label{SS:cplt}

We recall the following definition of completeness and a lemma from \cite{BaerBallmann12}.

\begin{definition}\label{D:cplt}
We call $D$ a \emph{complete} operator if the subspace of compactly supported sections in $\dom D_{\max}$ is dense in $\dom D_{\max}$ with respect to the graph norm of $D$.
\end{definition}

\begin{lemma}\cite[Lemma 3.1]{BaerBallmann12}\label{L:cplt}
Let $f:M\to\RR$ be a Lipschitz function with compact support and $u\in\dom D_{\max}$. Then $fu\in\dom D_{\max}$ and
\[
D_{\max}(fu)\;=\;\sigma_D(df)u+fD_{\max}u.
\]
\end{lemma}

The next theorem, again from \cite{BaerBallmann12}, is still true here with minor changes of the proof.

\begin{theorem}\label{T:cplt}
Let $D:C^\infty(M,E)\to C^\infty(M,F)$ be a differential operator of first order. Suppose that there exists a constant $C>0$ such that
\[
|\sigma_D(\xi)|\;\le\;C\,|\xi|
\]
for all $x\in M$ and $\xi\in T_x^*M$. Then $D$ and $D^*$ are complete.
\end{theorem}

\begin{proof}[Sketch of the proof]
Fix a base point $x_0\in\dm$ and let $r:M\to\RR$ be the distance function from $x_0$, $r(x)={\rm dist}(x,x_0)$. Then $r$ is a Lipschitz function with Lipschitz constant 1. Now the proof is exactly the same as that of \cite[Theorem 3.3]{BaerBallmann12}.
\end{proof}

\begin{example}\label{EX:Dirac}
If $D$ is a Dirac-type operator (cf. Subsection \ref{SS:Drac op}), then $\sigma_D(\xi)=\sigma_{D^*}(\xi)=c(\xi)$ is the Clifford multiplication. So one can choose $C=1$ in Theorem \ref{T:cplt} and therefore $D$ and $D^*$ are complete.
\end{example}

\section{Strongly Callias-type operators and their domains}\label{S:strCallias&maxdom}

In this section we introduce our main object of study -- strongly Callias-type operators. The main property of these operators is the discreteness of their spectra. We discuss natural domains for a strongly  Callias-type operator on a manifold with non-compact boundary. We also introduce a scale of Sobolev spaces defined by a strongly Callias-type operator.

\subsection{A Dirac operator}\label{SS:Drac op}

Let $M$ be a complete Riemannian manifold and let $E\to M$ be a Hermitian vector bundle over $M$. We use the Riemannian metric of $M$ to identify the tangent and the cotangent bundles, $T^*M\simeq TM$.

\begin{definition}[\cite{LawMic89}, Definition~II.5.2]\label{D:Dirac bundle}
The bundle $E$ is called a {\em Dirac bundle} over $M$ if the following data is given
\begin{enumerate}
\item a Clifford multiplication $c:TM\simeq T^*M\to \End(E)$, such that $c(\xi)^2= -|\xi|^2$ and $c(\xi)^*= -c(\xi)$ for every $\xi\in T^*M$;
\item a Hermitian connection $\n^E$ on $E$ which is compatible with the Clifford multiplication in the sense that 
\[
	\n^E\,\big(\, c(\xi)\, u\,\big)\ = \ c(\n^{LC}\xi)\,u\ + \ c(\xi)\,\n^Eu,\qquad u\in C^\infty(M,E).
\]
Here $\n^{LC}$ denotes the Levi-Civita connection on $T^*M$. 
\end{enumerate}
\end{definition}

If $E$ is a Dirac bundle we consider the {\em Dirac operator} $D:C^\infty(M,E)\to C^\infty(M,E)$ defined by 
\begin{equation}\label{E:Diracopr}
	D\;=\;\sum_{j=1}^n\,\,c(e_j)\,\nabla_{e_j}^E,
\end{equation}
where $e_1,\dots,e_n$ is an orthonormal basis of $TM\simeq T^*M$. One easily checks that $D$ is formally self-adjoint, $D^*= D$. 

\subsection{Strongly Callias-type operators}\label{strCallias}

Let $\Phi\in{\rm End}(E)$ be a self-adjoint bundle map (called a \emph{Callias potential}). Then 
\begin{equation}\label{E:Callias}
	\DD\ :=\ D\,+\,i\Phi 
\end{equation}
is a Dirac-type operator on $E$ and
\begin{equation}\label{E:Calliaseq}
	\DD^*\DD\;=\;D^2+\Phi^2+i[D,\Phi], \qquad
	\DD\DD^*\;=\;D^2+\Phi^2-i[D,\Phi],
\end{equation}
where $[D,\Phi]:=D\Phi-\Phi D$ is the commutator of the operators $D$ and $\Phi$.

\begin{definition}\label{D:strCallias}
We say that $\DD$ is a \emph{strongly Callias-type operator} if
\begin{enumerate}
\item $[D,\Phi]$ is a zeroth order differential operator, i.e. a bundle map;
\item for any $R>0$, there exists a compact subset $K_R\subset M$ such that
\begin{equation}\label{E:strinvinf}
	\Phi^2(x)\;-\;\big|[D,\Phi](x)\big|\;\ge\;R
\end{equation}
for all $x\in M\setminus K_R$. Here $\big|[D,\Phi](x)\big|$ denotes the operator norm of the linear map $[D,\Phi](x):E_x\to E_x$. In this case, the compact set $K_R$ is called an \emph{R-essential support} of $\DD$.
\end{enumerate}
A compact set $K\subset M$ is called an {\em essential support} of $\DD$ if there exists an $R>0$ such that $K$ is an $R$-essential support of $\DD$. 
\end{definition}

\begin{remark}\label{R:strCalliasequiv}
This is a stronger version of the Callias condition, \cite[Definition~1.1]{Anghel93Callias}. Basically, we require that the Callias potential grows to infinity at the infinite ends of the manifold. Note that $\DD$ is a strongly Callias-type operator if and only if $\DD^*$ is.
\end{remark}

\begin{remark}\label{R:Phi commutes}
Condition (i) of Definition~\ref{D:strCallias} is equivalent to the condition that $\Phi$ commutes with the Clifford multiplication
\begin{equation}\label{E:Phi commutes}
	\big[c(\xi),\Phi\big] \ = \ 0, \qquad\text{for all }\ \xi\in T^*M.
\end{equation}
\end{remark}

\subsection{A product structure}\label{SS:product	}

We say that the Riemannian metric $g^M$ is  {\em product near the boundary} if there exists a neighborhood $U\subset M$ of the boundary which is isometric to the cylinder
\begin{equation}\label{E:Zr}
	Z_r\ := \ [0,r)\times\dm \ \subset \ M. 
\end{equation} 
In the following we identify $U$ with $Z_r$ and denote by $t$ the coordinate along the axis of $Z_r$. Then the inward unit normal vector to the boundary is given by $\tau = dt$. 

Further, we assume that the Clifford multiplication $c:T^*M\to \End(E)$ and the connection  $\n^E$ also have product structure on $Z_r$. In this situation we say that the Dirac bundle $E$ is {\em product on $Z_r$}. We say that the Dirac bundle $E$ is  {\em product near the boundary} if there exists $r>0$, a neighborhood $U$ of $\pM$  and an isometry $U\simeq Z_r$ such that $E$ is product on $Z_r$. In this situation the  restriction of the Dirac operator to $Z_r$ takes the form 
\begin{equation}\label{E:productD}
	D\; = \; c(\tau)\,\big(\partial_t+A\big), 
\end{equation}
where, by \eqref{E:Diracopr} (with $\tau=e_n$), 
\[
	A\;=\; -\,\sum_{j=1}^{n-1}\,c(\tau)c(e_j)\nabla_{e_j}^E.
\]
The operator $A$ is formally self-adjoint $A^*=A$ and anticommutes with $c(\tau)$
\begin{equation}\label{E:anticommuting}
	A\circ c(\tau)\;=\;-c(\tau)\circ A.
\end{equation}

Let $\DD=D+i\Phi:C^\infty(M,E)\to C^\infty(M,E)$ be a strongly Callias-type operator. Then the restriction of $\DD$ to $Z_r$ is given by 
\begin{equation}\label{E:productDD}
	\DD\;=\; 
	c(\tau)\,\big(\,\partial_t+A-ic(\tau)\Phi\,\big)\;=\;c(\tau)\,\big(\partial_t+\AAA\big),
\end{equation}
where 
\begin{equation}\label{E:AAA}
		\AAA \ := \ A \ - \ ic(\tau)\Phi:\,C^\infty(\dm,E_{\dm})
		\ \to \ C^\infty(\dm,E_{\dm}).
\end{equation}

\begin{definition}\label{D:producDD}
We say that a Callias-type operator $\DD$ is {\em product near the boundary} if the Dirac bundle $E$ is  product near the boundary and the restriction of the Callias potential $\Phi$ to $Z_r$ does not depend on $t$. The operator $\AAA$  of \eqref{E:AAA} is called the {\em restriction of $\DD$ to the boundary}. 
\end{definition}

\subsection{The restriction of the adjoint to the boundary}\label{SS:DD*}

Recall that $\Phi$ is a self-adjoint bundle map, which, by Remark~\ref{R:Phi commutes}, commutes with the Clifford multiplication. It follows from \eqref{E:productDD}, that
\begin{equation}\label{E:productDD*}
	\DD^*\;=\; 
	c(\tau)\,\big(\,\partial_t+\AAA^\#\,\big) \ = \  
	c(\tau)\,\big(\,\partial_t+A+ic(\tau)\Phi\,\big),
\end{equation}
where 
\begin{equation}\label{E:AAA*}
	\AAA^\# \ := \ A \ + \ ic(\tau)\Phi.
\end{equation}
Thus, {\em $\DD^*$ is product near the boundary}.

From \eqref{E:Phi commutes} and \eqref{E:anticommuting}, we obtain
\begin{equation}\label{E:tildeA}
	\AAA^\#\ = \ - c(\tau)\circ \AAA\circ c(\tau)^{-1}.
\end{equation}

\subsection{Self-adjoint strongly Callias-type operators}\label{SS:sa_stronglyCallias}

Notice that $\AAA$ is a formally self-adjoint Dirac-type operator on $\dm$ and thus is an essentially self-adjoint elliptic operator by \cite[Theorem 1.17]{GromovLawson83}. Since $c(\tau)$ anticommutes with $A$, we have
\begin{equation}\label{E:squareAAA}
	\AAA^2\;=\;A^2+ic(\tau)[A,\Phi]+\Phi^2.
\end{equation}
It follows from Definition \ref{D:strCallias} and \eqref{E:AAA} that $[A,\Phi]$ is also a bundle map with the same norm as $[D,\Phi]$. Thus the last two terms on the right hand side of \eqref{E:squareAAA} grow to infinity at the infinite ends of $\dm$. By \cite[Lemma 6.3]{Shubin99}, the spectrum of $\AAA$ is discrete.  In this sense, $\AAA$ is very similar to a strongly Callias-type operator, with the difference that the potential $c(\tau)\Phi$ is {\em anti-self-adjoint} and, as a result $\AAA$ is self-adjoint. We formalize the properties of $\AAA$ in the following.

\begin{definition}\label{D:sa_stronglyCallias}
Let $A$ be a Dirac operator on $\pM$. An operator  $\AAA:= A+\Psi$, where $\Psi:E_\pM\to E_\pM$ is a self-adjoint bundle map, is called a  \emph{self-adjoint strongly Callias-type operator} if
\begin{enumerate}
\item the anticommutator $[A,\Psi]_+:= A\circ\Psi+\Psi\circ A$ is a zeroth order differential operator, i.e. a bundle map;
\item for any $R>0$, there exists a compact subset $K_R\subset \pM$ such that
\begin{equation}\label{E:strinvinfA}
	\Psi^2(x)\;-\;\big|[A,\Psi]_+(x)\big|\;\ge\;R
\end{equation}
for all $x\in \pM\setminus K_R$.  In this case, the compact set $K_R$ is called an \emph{R-essential support} of $\AAA$.
\end{enumerate}
\end{definition}

Using this definition we summarize the properties of $\AAA$ in the following.
\begin{lemma}\label{L:restrictiontobndry}
Let $\DD$ be a strongly Callias-type operator on a complete Riemannian manifold $M$ and let $\AAA$ be the restriction of\/ $\DD$ to the boundary. Then $\AAA$ is a self-adjoint strongly Callias-type operator on $\pM$. In particular, it has a discrete spectrum. 
\end{lemma}

\subsection{Sobolev spaces on the boundary}\label{SS:Sob&specproj}
The operator $\id+\AAA^2$ is positive. Hence, for any $s\in\RR$, its powers $(\id+\AAA^2)^{s/2}$ can be defined using functional calculus.

\begin{definition}\label{D:Sobolev}
Set
\[
	C_\AAA^\infty(\dm,E_{\dm})	\;:=\;
	\Big\{\,\bfu\in C^\infty(\dm,E_{\dm}):\,
	\big\|(\id+\AAA^2)^{s/2}\bfu\big\|_{L^2(\dm,E_\pM)}^2<+\infty\mbox{ for all }s\in\RR\,\Big\}.
\]
For all $s\in\RR$ we define the \emph{Sobolev $H_\AAA^s$-norm} on $C_\AAA^\infty(\dm,E_{\dm})$ by
\begin{equation}\label{E:Sobnorm}
	\|\bfu\|_{H_\AAA^s(\dm,E_\pM)}^2\;:=\;
		\big\|(\id+\AAA^2)^{s/2}\bfu\big\|_{L^2(\dm,E_\pM)}^2.
\end{equation}
The Sobolev space $H_\AAA^s(\dm,E_{\dm})$ is defined to be the completion of $C_\AAA^\infty(\dm,E_{\dm})$ with respect to this norm.
\end{definition}

\begin{remark}\label{R:Soblev}
In general, 
\[
	C_c^\infty(\dm,E_{\dm})\;\subset\; C_\AAA^\infty(\dm,E_{\dm})\;\subset \;
	C^\infty(\dm,E_{\dm}).
\] 
When $\dm$ is compact, the above spaces are all equal and the space $C_\AAA^\infty(\dm,E_{\dm})$ is independent of $\AAA$. However, if $\pM$ is not compact, these spaces are different and $C_\AAA^\infty(\dm,E_{\dm})$ does depend on the operator $\AAA$. Consequently,  if $\pM$ is not compact, the  Sobolev spaces $H_\AAA^s(\dm,E_{\dm})$  depend on $\AAA$.
\end{remark}

\begin{remark}\label{R:alternative Sobolev}
Alternatively one could define the $s$-Sobolev space to be the completion of\/ $C_c^\infty(\dm,E_{\dm})$ with respect to the $H_\AAA^s$-norm. In general, this leads to a different scale of Sobolev spaces, cf. \cite[\S3.1]{Hebey99book} for more details. We prefer our definition, since the space $H^{\rm fin}(\AAA)$,  defined below in \eqref{E:Hfin}, which plays an important role in our discussion, is a subspace of $C^\infty_\AAA(\pM,E_\pM)$ but is not a subspace of $C_c^\infty(\pM,E_\pM)$. 
\end{remark}

The rest of this section follows rather closely the exposition in Sections~5 and 6 of \cite{BaerBallmann12} with some changes needed to accommodate the non-compactness of the boundary. 

\subsection{Eigenvalues and eigensections of $\AAA$}\label{SS:eigenvalues}

Let
\begin{equation}\label{E:spec}
-\infty\leftarrow\cdots\le\lambda_{-2}\le\lambda_{-1}\le\lambda_0\le\lambda_1\le\lambda_2\le\cdots\to+\infty
\end{equation}
be the spectrum of $\AAA$ with each eigenvalue being repeated according to its (finite) multiplicity. Fix a corresponding $L^2$-orthonormal basis $\{\bfu_j\}_{j\in\ZZ}$ of eigensections of $\AAA$. By definition, each element in $C_\AAA^\infty(\dm,E_{\dm})$ is $L^2$-integrable and thus can be written as $\bfu=\sum_{j=-\infty}^\infty a_j\bfu_j$. Then
\[
\|\bfu\|_{H_\AAA^s(\dm,E_\pM)}^2\;=\;\sum_{j=-\infty}^\infty\,|a_j|^2(1+\lambda_j^2)^s.
\]
On the other hand, let
\begin{equation}\label{E:Hfin}
		H^{\rm fin}(\AAA)\;:=\;
	\Big\{\,\bfu=\sum_ja_j\bfu_j:\,\,a_j=0\mbox{ for all but finitely many }j\,\Big\}
\end{equation}
be the space of finitely generated sections. Then $H^{\rm fin}(\AAA)\subset C_\AAA^\infty(\dm,E_{\dm})$ and for any $s\in\RR$, $H^{\rm fin}(\AAA)$ is dense in $H_\AAA^s(\dm,E_{\dm})$. We obtain an alternative description of the Sobolev spaces
\[
	H_\AAA^s(\dm,E_{\dm})\;=\;
	\Big\{\,\bfu=\sum_ja_j\bfu_j:\,\,\sum_j|a_j|^2(1+\lambda_j^2)^s<+\infty\,\Big\}.
\]

\begin{remark}\label{R:Sobolev}
The following properties follow from our definition and preceding discussion.
\begin{enumerate}
\item $H_\AAA^0(\dm,E_{\dm})=L^2(\dm,E_{\dm})$.
\item\label{I:embedding} If $s<t$, then $\|\bfu\|_{H_\AAA^s(\dm,E_\pM)}\le\|\bfu\|_{H_\AAA^t(\dm,E_\pM)}$. And we shall show shortly in Theorem \ref{T:cptembedding} that there is still a Rellich embedding theorem, i.e., the induced embedding $H_\AAA^{t}(\dm,E_{\dm}) \hookrightarrow H_\AAA^s(\dm,E_{\dm})$ is compact.
\item $\bigcap_{s\in\RR}H_\AAA^s(\dm,E_{\dm})=C_\AAA^\infty(\dm,E_{\dm})$.
\item\label{I:pairing} For all $s\in\RR$, the pairing
\[
H_\AAA^s(\dm,E_{\dm})\;\times\;H_\AAA^{-s}(\dm,E_{\dm})\;\to\;\CC,
\qquad 
\Big(\sum_j a_j\bfu_j,\sum_j b_j\bfu_j\Big)\;\mapsto\;\sum_j\overline{a}_jb_j
\]
is perfect. Therefore, $H_\AAA^s(\dm,E_{\dm})$ and $H_\AAA^{-s}(\dm,E_{\dm})$ are pairwise dual.
\end{enumerate}
\end{remark}

We have the following version of the Rellich Embedding Theorem:
\begin{theorem}\label{T:cptembedding}
For $s<t$, the embedding $H_\AAA^{t}(\dm,E_{\dm})\hookrightarrow H_\AAA^s(\dm,E_{\dm})$ mentioned in Remark \ref{R:Sobolev}.\ref{I:embedding} is compact.
\end{theorem}

To prove the theorem, we use the following result,  cf., for example,  \cite[Proposition 2.1]{Bisgard12}.

\begin{proposition}\label{P:cptembedding}
A closed bounded subset $K$ in a Banach space $X$ is compact if and only if for every $\varepsilon>0$, there exists a finite dimensional subspace $Y_\varepsilon$ of $X$ such that every element $x\in K$ is within distance $\varepsilon$ from $Y_\varepsilon$.
\end{proposition}

\begin{proof}[Proof of Theorem \ref{T:cptembedding}]
Let $B$ be the unit ball in $H_\AAA^{t}(\dm,E_{\dm})$. We use Proposition~\ref{P:cptembedding} to show that the closure $\bar{B}$ of $B$ in $H_\AAA^{s}(\dm,E_{\dm})$ is compact in $H_\AAA^{s}(\dm,E_{\dm})$.

For simplicity, suppose that $\lambda_0$ is an eigenvalue of $\AAA$ with smallest absolute value and for $n>0$, set $\Lambda_n:=\min\{\lambda_n^2,\lambda_{-n}^2\}$. Then $\{\Lambda_n\}$ is an increasing sequence by \eqref{E:spec}. For every $\varepsilon>0$, there exists an integer $N>0$, such that $(1+\Lambda_n)^{s-t}<\varepsilon^2/4$ for all $n\ge N$. 

Consider the finite-dimensional space
\[
	Y_\varepsilon\;:=\;
	{\rm span}\{\bfu_{j}:-N\le j\le N\}\;\subset\;H_\AAA^{s}(\dm,E_{\dm}).
\]
We claim that every element $\bar{\bfu}\in \bar{B}$ is within distance $\varepsilon$ from $Y_\varepsilon$. Indeed, choose $\bfu=\sum_ja_j\bfu_j\in B$, such that the $H_\AAA^{s}$-distance between $\bar{\bfu}$ and $\bfu$ is less than $\varepsilon/2$. Then $\bfu':=\sum_{j=-N}^Na_j\bfu_j$ belongs to $Y_\varepsilon$ and the $H_\AAA^{s}$-distance
\begin{multline}\notag
	\|\bfu-\bfu'\|_{H_\AAA^{s}(\dm,E_\pM)}^2 \;=\;
	\sum_{|j|>N}|a_j|^2(1+\lambda_j^2)^s\
	\;\le\;\sum_{|j|>N}|a_j|^2(1+\lambda_j^2)^t\cdot(1+\Lambda_N)^{s-t}
	\\ \le \ 
	\|\bfu\|_{H_\AAA^{t}(\dm,E_\pM)}^2\cdot (1+\Lambda_N)^{s-t} 
	\;\le\;
	(1+\Lambda_N)^{s-t}\;<\;\frac{\varepsilon^2}{4}.
\end{multline}
Hence $\bfu$ is within distance $\varepsilon/2$ of $Y_\varepsilon$, and therefore $\bar{\bfu}$ is within distance $\varepsilon$ of $Y_\varepsilon$. The theorem then follows from Proposition \ref{P:cptembedding}.
\end{proof}

\subsection{The hybrid Soblev spaces}\label{SS:hybrid}
For $I\subset\RR$, let
\begin{equation}\label{E:specproj}
	P_I^\AAA\;:\;\sum_j a_j\bfu_j\;\mapsto\;\sum_{\lambda_j\in I}a_j\bfu_j
\end{equation}
be the spectral projection. It's easy to see that
\[
	H_I^s(\AAA)\;:=\;
	P_I^\AAA(H_\AAA^s(\dm,E_{\dm}))\;\subset\;H_\AAA^s(\dm,E_{\dm})
\]
for all $s\in\RR$. 

\begin{definition}\label{D:hybrid space}
For $a\in\RR$, we define the \emph{hybrid} Sobolev space
\begin{equation}\label{E:checkH}
	\cHH(\AAA)\;:=\;
	H_{(-\infty,a]}^{1/2}(\AAA)\;\oplus\;H_{(a,\infty)}^{-1/2}(\AAA)
	\ \subset \ H^{-1/2}_\AAA(\dm,E_{\dm})
\end{equation}
with $\cHH$-norm
\[
	\|\bfu\|_{\cHH(\AAA)}^2\;:=\;
	\big\|P_{(-\infty,a]}^\AAA\bfu\big\|_{H_\AAA^{1/2}(\dm,E_\pM)}^2\;+\;
	\big\|P_{(a,\infty)}^\AAA\bfu\big\|_{H_\AAA^{-1/2}(\dm,E_\pM)}^2.
\]
\end{definition}

The space $\cHH(\AAA)$ is independent of the choice of $a$. Indeed, for $a_1<a_2$, the difference between the corresponding $\cHH$-norms only occurs on the finite dimensional space  $P_{[a_1,a_2]}^\AAA(L^2(\dm,E_{\dm}))$. Thus the norms defined using different values of $a$ are equivalent.

Similarly, we define
\begin{equation}\label{E:hatH}
	\hHH(\AAA)\; := \;
	H_{(-\infty,a]}^{-1/2}(\AAA)\;\oplus\;H_{(a,\infty)}^{1/2}(\AAA)
\end{equation}
with $\hHH$-norm
\[
\|\bfu\|_{\hHH(\AAA)}^2\;:=\;
\|P_{(-\infty,a]}^\AAA\bfu\|_{H_\AAA^{-1/2}(\dm,E_\pM)}^2\;+\;
\|P_{(a,\infty)}^\AAA\bfu\|_{H_\AAA^{1/2}(\dm,E_\pM)}^2.
\]
Then 
\[
	\hHH(\AAA)\;=\;\cHH(-\AAA).
\]
The pairing of Remark \ref{R:Sobolev}.\ref{I:pairing} induces a perfect pairing
\[
\cHH(\AAA)\;\times\;\hHH(\AAA)\;\to\;\CC.
\]

\subsection{The hybrid space of the dual operator}\label{SS:hybrid dual}

Recall,  that the restriction $\AAA^\#$ of $\DD^*$ to the boundary can be computed by \eqref{E:tildeA}. Thus 
the isomorphism $c(\tau):E_{\dm}\to E_{\dm}$ sends each eigensection $\bfu_j$ of $\AAA$ associated to eigenvalue $\lambda_j$ to an eigensection of $\AAA^\#$ associated to eigenvalue $-\lambda_j$. We conclude that the set of eigenvalues of $\AAA^\#$ is $\{-\lambda_j\}_{j\in \ZZ}$ with associated $L^2$-orthonormal eigensections $\{c(\tau)\bfu_j\}_{j\in\ZZ}$. For $\bfu=\sum_ja_j\bfu_j\in H_\AAA^s(\dm,E_{\dm})$, we have 
\[
	\|c(\tau)\bfu\|_{H_{\AAA^\#}^s(\dm,E_\pM)}^2
	\;=\;
	\sum_j|a_j|^2\,\big(1+(-\lambda_j)^2\big)^s\;=\;\|\bfu\|_{H_\AAA^s(\dm,E_\pM)}^2.
\]
So $c(\tau)$ induces an isometry between Sobolev spaces $H_\AAA^s(\dm,E_{\dm})$ and $H_{\AAA^\#}^s(\dm,E_{\dm})$ for any $s\in\RR$. Furthermore, it restricts to an  isomorphism between $H_{(-\infty,a]}^s(\AAA)$  and $H_{[-a,\infty)}^s(\AAA^\#)$. Therefore we conclude that

\begin{lemma}\label{L:admap}
Over $\dm$, the isomorphism $c(\tau):E_{\dm}\to E_{\dm}$ induces an isomorphism $\hHH(\AAA)\to\cHH(\AAA^\#)$. In particular, the sesquilinear form
\[
	\beta:\cHH(\AAA)\times\cHH(\AAA^\#) \ \to \ \CC,
	\qquad\	\beta(\bfu,\bfv) \ :=\ -\,\big(\bfu,-c(\tau)\bfv\big)
	\ =\ -\,\big(c(\tau)\bfu,\bfv\big),
\]
is a perfect pairing of topological vector spaces.
\end{lemma}

\subsection{Sections in a neighborhood of the boundary}\label{SS:sections in Zr}

Recall from \eqref{E:Zr} that we identify a neighborhood of $\pM$ with the product $Z_r=[0,r)\times\pM$. The $L^2$-sections over $Z_r$ can be written as
\[
	u(t,x)\;=\;\sum_j\,  a_j(t)\,\bfu_j(x)
\]
in terms of the $L^2$-orthonormal basis $\{\bfu_j\}$ on $\dm$. We fix a smooth cut-off function $\chi:\RR\to\RR$ with
\begin{equation}\label{E:cut-off}
	\chi(t) \ = \ 
	\begin{cases}1\ \   &\mbox{for} \ t\le r/3\\
	0\ \  &\mbox{for} \ t\ge2r/3.
	\end{cases}
\end{equation}

Recall that $H^{\rm fin}(\AAA)$ is dense in $\cHH(\AAA)$ and $\hHH(\AAA)$. For $\bfu\in H^{\rm fin}(\AAA)$, we define a smooth section $\EE\bfu$ over $Z_r$ by
\begin{equation}\label{E:extension}
	(\EE\bfu)(t)\;:=\;\chi(t)\cdot\exp(-t|\AAA|)\bfu.
\end{equation}
Thus, if $\bfu(x)=\sum_ja_j\bfu_j(x)$, then
\begin{equation}\label{E:extension2}
	(\EE\bfu)(t,x)\;=\;\chi(t)\,\sum_ja_j\cdot\exp(-t|\lambda_j|)\cdot\bfu_j(x).
\end{equation}
It's easy to see that $\EE\bfu$ is an $L^2$-section over $Z_r$. So we get a linear map
\[
	\EE:\;H^{\rm fin}(\AAA)\;\to\;C^\infty(Z_r,E)\cap L^2(Z_r,E)
\]
which we call the {\em extension map}.

As in Subsection~\ref{SS:minmaxext} we denote by $\|\cdot\|_\DD$ the graph norm of $\DD$.

\begin{lemma}\label{L:extension}
For all\/ $\bfu\in H^{\rm fin}(\AAA)$, the extended section $\EE\bfu$ over $Z_r$ belongs to $\dom\DD_{\max}$. And there exists a constant $C=C(\chi,\AAA)>0$ such that
\[
	\big\|\EE\bfu\big\|_\DD\;\le\;C\,\|\bfu\|_{\cHH(\AAA)}
	\quad\mbox{and}\quad
	\big\|c(\tau)\EE\bfu\big\|_{\DD^*}\;\le\;C\,\|\bfu\|_{\hHH(\AAA)}.
\]
\end{lemma}

\begin{proof}
For the first claim, we only need to show that $\DD(\EE\bfu)$ is an $L^2$-section over $Z_r$. Since
\[
	\DD(\EE\bfu)\;=\;
	\DD(\EE P_{(-\infty,0]}^\AAA\bfu)\;+\;\DD(\EE P_{(0,\infty)}^\AAA\bfu),
\]
it suffices to consider each summand separately. Recall that $\DD=c(\tau)(\partial_t+\AAA)$ on $Z_r$. By \eqref{E:extension}, we have
\[
	\DD(\EE P_{(0,\infty)}^\AAA\bfu)\;=\;
	c(\tau)\,\chi'\exp(-t\AAA)P_{(0,\infty)}^\AAA\bfu,
\]
which is clearly an $L^2$-section over $Z_r$. On the other hand,
\[
	\DD(\EE P_{(-\infty,0]}^\AAA\bfu)\;=\;
	c(\tau)\,\big(2\chi\AAA+\chi'\big)\exp(t\AAA)P_{(-\infty,0]}^\AAA\bfu,
\]
which is again an $L^2$-section over $Z_r$. Therefore $\EE\bfu\in\dom\DD_{\max}$.

The proof of the first inequality is exactly the same as that of \cite[Lemma 5.5]{BaerBallmann12}. For the second one, just notice that $\AAA^\#$ is the restriction to the boundary of $\DD^*$ and, by Lemma~\ref{L:admap}, $c(\tau):\cHH(\AAA^\#)\to \hHH(\AAA)$ is an isomorphism of Hilbert spaces.
\end{proof}

The following lemma is an analogue of \cite[Lemma 6.2]{BaerBallmann12} with exactly the same proof.

\begin{lemma}\label{L:restriction}
There is a constant $C>0$ such that for all $u\in C_c^\infty(Z_r,E)$,
\[
\|u_{\dm}\|_{\cHH(\AAA)}\;\le\;C\,\|u\|_\DD.
\]
\end{lemma}

\subsection{A natural domain for boundary value problems}\label{SS:natural domain}

For closed manifolds the ellipticity of $\DD$ implies that $\dom(\DD_{\max})\subset \Hloc^1(M,E)$. However, if $\pM\not=\emptyset$, then near the boundary the sections in $\dom(\DD_{\max})$ can behave badly. That is why, if one wants to talk about boundary value of sections,  one needs to consider a smaller domain for $\DD$.

\begin{definition}\label{D:H1D}
We define the norm
\begin{equation}\label{E:locH1}
	\|u\|_{H_\DD^1(Z_r,E)}^2\;:=\;
	\|u\|_{L^2(Z_r,E)}^2\;+\;\|\partial_tu\|_{L^2(Z_r,E)}^2\;+\;
	\|\AAA u\|_{L^2(Z_r,E)}^2.
\end{equation}
and denote by $H_\DD^1(Z_r,E)$ the completion of  $C_c^\infty(Z_r,E)$ with respect to this norm. We refer to \eqref{E:locH1} as the $H_\DD^1(Z_r)$-norm. 

In general, for any integer $k\ge1$, let $H_\DD^k(Z_r,E)$ be the completion of $C_c^\infty(Z_r,E)$ with respect to the $H_\DD^k(Z_r)$-norm given by
\begin{equation}\label{E:locHk}
	\|u\|_{H_\DD^k(Z_r,E)}^2\;:=\;
	\|u\|_{L^2(Z_r,E)}^2\;+\;\|(\partial_t)^ku\|_{L^2(Z_r,E)}^2
	\;+\;\|\AAA^ku\|_{L^2(Z_r,E)}^2.
\end{equation}
\end{definition}

Note that $H_\DD^1(Z_r,E)\subset H^1_{\loc}(Z_r,E)\cap L^2(Z_r,E)$. Moreover, we have the following analogue of the Rellich embedding theorem:

\begin{lemma}\label{L:comHDtoL2}
The inclusion map $H_\DD^1(Z_r,E)\hookrightarrow L^2(Z_r,E)$ is compact. 
\end{lemma}

\begin{proof}
Let $B$ be the unit ball about the origin in $H_\DD^1(Z_r,E)$ and let $\oB$ denote its closure in $L^2(Z_r,E)$. We need to prove that $\oB$ is compact. By Proposition~\ref{P:cptembedding} it is enough to show that for every $\varepsilon>0$ there exists a finite dimensional subspace $Y_\varepsilon\in L^2(Z_r,E)$ such that every $u\in \oB$ is within distance $\varepsilon$ from $Y_\varepsilon$.

Let $\lambda_j$ and $\bfu_j$ be as in Subsection~\ref{SS:eigenvalues}.
As in the proof of Theorem~\ref{T:cptembedding} we set $\Lambda_n:=\min\{\lambda_n^2,\lambda_{-n}^2\}$. Choose $N>0$ such that 
\begin{equation}\label{E:Lambda}
	1+\Lambda_n\ > \ \frac{8}{\varepsilon^2} \qquad\text{for all}\quad  n\ge N.
\end{equation} 

Let $H^1([0,r))$ denote the Sobolev space of complex-valued functions on the interval $[0,r)$ with norm 
\[
	\|a\|_{H^1([0,r))}^2\ :=\ \|a\|^2_{L^2([0,r))}\ + \ \|a'\|^2_{L^2([0,r))}.
\]
Let $B'\subset H^1([0,r))$ denote the unit ball about the origin in $H^1([0,r))$ and let $\oB'$ be its closure in $L^2([0,r))$. By the classical Rellich embedding theorem $\oB'$ is compact in $L^2([0,r))$. Hence, for every $\varepsilon>0$ there exists a finite set $X_{\varepsilon}$ such that every $a\in \oB'$ is within distance $\frac\varepsilon{\sqrt{16N+8}}$ from $X_{\varepsilon}$.

We now define the finite dimensional space
\[
	Y_\varepsilon \ := \ 
	\big\{\sum_{j=-N}^N\, a_j(t)\,\bfu_j:\, a_j(t)\in X_{\varepsilon}\,\big\}
	\ \subset \ L^2(Z_r,E).
\]
We claim that every $u\in \oB$ is within distance $\varepsilon$ from $Y_\varepsilon$. Indeed, let $\ou\in \oB$. We choose $u= \sum_{j=-\infty}^\infty b_j(t)\bfu_j\in B$ such that 
\begin{equation}\label{E:ou-u}
	\|\ou-u\| \ < \ \frac{\varepsilon}2. 
\end{equation}
Since $\{\bfu_j\}$ is an orthonormal basis of $L^2(\pM,E_\pM)$, we conclude from \eqref{E:locH1} that 
\[
	\|u\|^2_{H^1_\DD(Z_r,E)} \ = \ 
	\sum_{j=-\infty}^\infty\, \Big(\,
	(1+\lambda_j^2)\,\|b_j\|^2_{L^2([0,r))}+ \|b_j'\|^2_{L^2([0,r))}\,\Big).
\]
Since $\|u\|^2_{H^1_\DD(Z_r,E)}\le1$, for all $j\in \ZZ$
\[
	(1+\lambda_j^2)\,\|b_j\|^2_{L^2([0,r))}\,+\, \|b_j'\|^2_{L^2([0,r))} \ \le \ 1.
\]
Hence, 
\begin{align}
	\|b_j\|^2_{L^2([0,r))}+ \|b_j'\|^2_{L^2([0,r))} \ &\le \ 1 
	\quad\Longrightarrow\quad
	b_j\in \oB', \qquad \text{for all }\ j\in \ZZ;\label{E:bjinB}\\
	\sum_{|j|>N}\|b_j\|^2_{L^2([0,r))}\ &< \ \frac{\varepsilon^2}8,\label{E:bj<}
\end{align}
where in the second inequality we use \eqref{E:Lambda}.

From \eqref{E:bjinB} we conclude that for every $j\in \ZZ$, there exists $a_j\in X_\varepsilon$ such that 
\[
	\|b_j-a_j\|_{L^2([0,r))} \ \le \ \frac{\varepsilon}{\sqrt{16N+8}}.
\]
Hence,
\begin{equation}\label{E:aj-bj}
	\sum_{j=-N}^N\,\|b_j-a_j\|^2_{L^2([0,r))} \ \le \ 
	(2N+1)\,\frac{\varepsilon^2}{16N+8} \ = \ \frac{\varepsilon^2}{8}.
\end{equation}
Set $u':= \sum_{j=-N}^Na_j(t)\bfu_j\in Y_\varepsilon$. Then  from \eqref{E:bj<} and \eqref{E:aj-bj} we obtain
\begin{multline}\notag
	\|u-u'\|^2_{L^2(Z_r,E)}\ = \ 
	\big\|\sum_{|j|>N}\,b_j\bfu_j + 
	\sum_{j=-N}^N\,(b_j-a_j)\bfu_j \big\|^2_{L^2(Z_r,E)}
	\\ \le \ 
	\sum_{|j|>N}\,\|b_j\|^2_{L^2([0,r))} \ + \ 
	\sum_{j=-N}^N\,\|b_j-a_j\|^2_{L^2([0,r))}
	\ \le\ 
	\frac{\varepsilon^2}{8}\ + \ \frac{\varepsilon^2}{8}  \ = \ \frac{\varepsilon^2}4.
\end{multline}
Combining this with \eqref{E:ou-u} we obtain 
\[
	\|\ou-u'\|_{L^2(Z_r,E)} \ \le \ \|\ou-u\|_{L^2(Z_r,E)}\ + \ \|u-u'\|_{L^2(Z_r,E)}
	\ \le \ \frac{\varepsilon}{2}\ + \ \frac{\varepsilon}{2} \ = \ 
	\varepsilon,
\] 
i.e., $\ou$ is within distance $\varepsilon$ from $Y_\varepsilon$. 
\end{proof}

\begin{lemma}\label{L:normequiv}
For all $u\in C_c^\infty(Z_r,E)$ with $P_{(0,\infty)}^\AAA(u_{\dm})=0$, we have  estimate 
\begin{equation}\label{E:eqnorms}
		\frac1{\sqrt{2}}\,\|u\|_\DD\;\le\;\|u\|_{H_\DD^1(Z_r,E)}\;\le\;\|u\|_\DD.
\end{equation}
\end{lemma}

\begin{proof}
Since $\DD=c(\tau)(\partial_t+\AAA)$ on $Z_r$, we obtain
\[
   \|u\|_\DD^2 \ \le \ 
   \|u\|_{L^2(Z_r,E)}^2\ +\ 2\,\big(\,\|\p_tu\|_{L^2(Z_r,E)}^2 \ + \ \|\AAA u\|_{L^2(Z_r,E)}^2\,\big)
   \ \le\ 2\,\|u\|_{H_\DD^1(Z_r,E)}^2,
\]
for all $u\in C_c^\infty(Z_r,E)$. This proves the first inequality in \eqref{E:eqnorms}.

Suppose that $u\in C_c^\infty(Z_r,E)$ with $P_{(0,\infty)}^\AAA(u_{\dm})=0$. We want to show the converse inequality.

We can write $u=\sum_ja_j(t)\bfu_j$. Then $a_j(r)=0$ for all $j$ and $a_j(0)=0$ for all $j$ such that $\lambda_j>0$. The latter condition means that
\[
\sum_j\lambda_j\,|a_j(0)|^2\;\le\;0.
\]
Then
\begin{equation}\label{E:normequiv}
\begin{aligned}
\|\DD u\|_{L^2(Z_r,E)}^2
&=\sum_j\int_0^r|a_j'(t)+a_j(t)\lambda_j|^2dt\\
&=\sum_j\Big(\int_0^r|a_j'(t)|^2dt+\lambda_j^2\int_0^r|a_j(t)|^2dt+\lambda_j\int_0^r(a_j'(t)\bar{a_j}(t)+a_j(t)\bar{a_j'}(t))dt\Big)\\
&=\sum_j\Big(\int_0^r|a_j'(t)|^2dt+\lambda_j^2\int_0^r|a_j(t)|^2dt+\lambda_j\int_0^r\frac{d}{dt}|a_j(t)|^2dt\Big)\\
&=\sum_j\Big(\int_0^r|a_j'(t)|^2dt+\lambda_j^2\int_0^r|a_j(t)|^2dt+\lambda_j(|a_j(r)|^2-|a_j(0)|^2)\Big)\\
&\ge\sum_j\Big(\int_0^r|a_j'(t)|^2dt+\lambda_j^2\int_0^r|a_j(t)|^2dt\Big)\\
&=\|\partial_tu\|_{L^2(Z_r,E)}^2+\|\AAA u\|_{L^2(Z_r,E)}^2.
\end{aligned}
\end{equation}
Hence 
\begin{multline}\notag
	\|u\|_\DD^2 \ :=  \ \|u\|_{L^2(Z_r,E)}^2 + \|\DD u\|_{L^2(Z_r,E)}^2
	\\ \ge \ 
	\|u\|_{L^2(Z_r,E)}^2+\|\partial_tu\|_{L^2(Z_r,E)}^2+\|\AAA u\|_{L^2(Z_r,E)}^2
	\ =: \ \|u\|_{H_\DD^1(Z_r,E)}^2.
\end{multline}
\end{proof}

\begin{remark}\label{R:normequiv}
In particular, the two norms are equivalent on $C_{cc}^\infty(Z_r,E)$.
\end{remark}

\subsection{The trace theorem}\label{SS:trace}

The following ``trace theorem" establishes the relationship between $H_\DD^k(Z_r,E)$ and the Sobolev spaces on the boundary.

\begin{theorem}[The trace theorem]\label{T:tracethm}
For all $k\ge1$, the restriction map (or trace map) 
\[
	\cRR:\,C_c^\infty(Z_r,E)\ \to \ C_c^\infty(\dm,E_{\dm}), \qquad
	\cRR(u)\ :=\ u_{\dm}
\] 
extends to a continuous linear map
\[
\cRR:\;H_\DD^k(Z_r,E)\;\to\;H_\AAA^{k-1/2}(\dm,E_{\dm}).
\]
\end{theorem}

\begin{proof}
Let $u(t,x)=\sum_ja_j(t)\bfu_j(x)\in C_c^\infty(Z_r,E)$. Then $\cRR(u)=u_{\dm}(x)=\sum_ja_j(0)\bfu_j(x)$, and we want to show that
\begin{equation}\label{E:traceest}
	\|u_{\dm}\|_{H_\AAA^{k-1/2}(\dm,E_\pM)}^2\;\le\;
	C(k)\,\|u\|_{H_\DD^k(Z_r,E)}^2
\end{equation}
for some constant $C(k)>0$.

Applying inverse Fourier transform to $a_j(t)$ yields that
\[
	a_j(t)\;=\;\int_\RR\, e^{it\cdot\xi}\,\widehat{a_j}(\xi)\,d\xi,
\]
where $\widehat{a_j}(\xi)$ is the Fourier transform of $a_j(t)$. (Here we use normalized measure to avoid the coefficient $2\pi$.) So
\[
	a_j(0)\;=\;\int_\RR\,\widehat{a_j}(\xi)\,d\xi.
\]
By H\"{o}lder's inequality,
\begin{multline}\notag
	|a_j(0)|^2
	\ = \ 
	\Big(\int_\RR\,\widehat{a_j}(\xi)\,d\xi\Big)^2
	\ \le \ 
	\Big(\int_\RR\,|\widehat{a_j}(\xi)|\,
	  (1+\lambda_j^2+\xi^2)^{k/2}\,(1+\lambda_j^2+\xi^2)^{-k/2}\,d\xi\Big)^2\\
	\ \le \ 
	\int_\RR\,|\widehat{a_j}(\xi)|^2\,(1+\lambda_j^2+\xi^2)^kd\xi\cdot
		\int_\RR(1+\lambda_j^2+\xi^2)^{-k}\,d\xi,
\end{multline}
where $\lambda_j$ is the eigenvalue of $\AAA$ corresponding to index $j$. We do the substitution $\xi=(1+\lambda_j^2)^{1/2}\tau$ to get
\[
\int_\RR(1+\lambda_j^2+\xi^2)^{-k}d\xi=(1+\lambda_j^2)^{-k+1/2}\int_\RR(1+\tau^2)^{-k}d\tau.
\]
It's easy to see that the integral on the right hand side converges when $k\ge1$ and depends only on $k$. Therefore
\begin{equation}\label{E:normest}
\begin{aligned}
|a_j(0)|^2(1+\lambda_j^2)^{k-1/2}
&\le C_1(k)\int_\RR|\widehat{a_j}(\xi)|^2(1+\lambda_j^2+\xi^2)^kd\xi\\
&\le C(k)\int_\RR|\widehat{a_j}(\xi)|^2(1+\lambda_j^{2k}+\xi^{2k})d\xi\\
&\le C(k)\Big(\int_\RR|a_j(t)|^2dt+\int_\RR|\widehat{a_j}(\xi)|^2\xi^{2k}d\xi+\int_\RR|a_j(t)|^2\lambda_j^{2k}dt\Big),
\end{aligned}
\end{equation}
where we use Plancherel's identity from line 2 to line 3. Recall the differentiation property of Fourier transform $\widehat{(\partial_t)^ka_j(t)}(\xi)=\widehat{a_j}(\xi)\xi^k$. So again by Plancherel's identity
\[
	\int_\RR\,|\widehat{a_j}(\xi)|^2\,\xi^{2k}\,d\xi
	\ =\ 
	\int_\RR\,|\widehat{(\partial_t)^k\,a_j(t)}(\xi)|^2\,d\xi
	\ = \ \int_\RR\,|(\partial_t)^k\,a_j(t)|^2\,dt
\]
Now summing inequality \eqref{E:normest} over $j$ gives \eqref{E:traceest} and the theorem is proved.
\end{proof}

\subsection{The space $H_\DD^1(M,E)$}\label{SS:natural domain M}

Recall that the cut-off function $\chi$ is defined in \eqref{E:cut-off}. By a slight abuse of notation we also denote by $\chi$ the induced function on $M$. Define
\begin{equation}\label{E:HDM}
		H_\DD^1(M,E)\;:=\;\dom\DD_{\max}\;\cap\;
		\big\{\,u\in L^2(M,E):\,\chi u\in H_\DD^1(Z_r,E)\,\big\}.
\end{equation}
It is a Hilbert space with the $H_\DD^1$-norm
\[
	\|u\|_{H_\DD^1(M,E)}^2\;:=\;
	\|u\|_{L^2(M,E)}^2\;+\;\|\DD u\|_{L^2(M,E)}^2\;+\;\|\chi u\|_{H_\DD^1(Z_r,E)}^2.
\]
As one can see from Remark \ref{R:normequiv}, a different choice of the cut-off function $\chi$ leads to an equivalent norm.  The $H_\DD^1$-norm is stronger than the graph norm of\/ $\DD$ in the sense that it controls in addition the $H_\DD^1$-regularity near the boundary. We call it $H_\DD^1$-regularity as it depends on our concrete choice of the norm \eqref{E:locH1}, unlike the case in \cite{BaerBallmann12}, where the boundary is compact. 

Lemma \ref{L:normequiv} and Theorem \ref{T:tracethm} extend from $Z_r$ to $M$. By the definition of $H_\DD^1(M,E)$ and the fact that $\DD$ is complete, we have

\begin{lemma}\label{L:H1-regularity}
\begin{enumerate}
\item $C_c^\infty(M,E)$ is dense in $H_\DD^1(M,E)$;
\item\label{I:H1-regularity-2} $C_{cc}^\infty(M,E)$ is dense in $\{u\in H_\DD^1(M,E):u_{\dm}=0\}$.
\end{enumerate}
\end{lemma}

The following statement is an immediate consequence of Remark \ref{R:normequiv} and Lemma \ref{L:H1-regularity}.\ref{I:H1-regularity-2}.

\begin{corollary}\label{C:normequiv}
$\dom\DD_{\min}=\{u\in H_\DD^1(M,E):u_{\dm}=0\}$.
\end{corollary}

\subsection{Regularity of the maximal domain}\label{SS:maxdom}
We now  state the main result of this section which extends Theorem~6.7 of \cite{BaerBallmann12} to manifolds with non-compact boundary.

\begin{theorem}\label{T:maxdom}
Assume that $\DD$ is a strongly Callias-type operator. Then
\begin{enumerate}
\item $C_c^\infty(M,E)$ is dense in\/ $\dom\DD_{\max}$ with respect to the graph norm of\/ $\DD$.
\item The trace map $\cRR:\,C_c^\infty(M,E)\to C_c^\infty(\dm,E_{\dm})$ extends uniquely to a surjective bounded linear map\/ $\cRR:\,\dom\DD_{\max}\to\cHH(\AAA)$.
\item $H_\DD^1(M,E)=\{u\in\dom\DD_{\max}:\,\cRR u\in H_\AAA^{1/2}(\dm,E_{\dm})\}$.
\end{enumerate}
The corresponding statements hold for $\dom(\DD^*)_{\max}$ (with $\AAA$ replaced with $\AAA^\#$). Furthermore, for all sections $u\in\dom\DD_{\max}$ and $v\in\dom(\DD^*)_{\max}$, we have
\begin{equation}\label{E:genGreensfor}
	\big(\DD_{\max}u,v\big)_{L^2(M,E)}
	\;-\;
	\big(u,(\DD^*)_{\max}v\big)_{L^2(M,E)}\;=\;
	-\,\big(c(\tau)\cRR u,\cRR v\big)_{L^2(\dm,E_\pM)}.
\end{equation}
\end{theorem}

\begin{remark}\label{R:maxdom}
In particular, (ii) of Theorem \ref{T:maxdom} says that $C_c^\infty(\dm,E_{\dm})$ is dense in $\cHH(\AAA)$.
\end{remark}

\begin{proof}
The proof goes along the same line as the proof of Theorem~6.7 in \cite{BaerBallmann12} but some extra care is needed because of non-compactness of the boundary. 

(i) \ Let $\tilde M$ be the double of $M$ formed by gluing two copies of $M$ along their boundaries. Then $\tilde M$ is a complete manifold without boundary. One can extend the Riemannian metric $g^M$, the Dirac bundle $E$ and the Callias-type operator $\DD$ on $M$ to a Riemannian metric $g^{\tilde M}$, a Dirac bundle $\tilde E$ and a Callias-type operator $\tilde\DD$ on $\tilde M$. Notice that now $\dom\tilde\DD_{\max}=\dom\tilde\DD_{\min}$ by \cite{GromovLawson83}.

\begin{lemma}\label{L:ellipreg}
If\/ $\tilde u\in\dom\tilde\DD_{\max}$, then\/ $u:=\tilde u|_M\in H_\DD^1(M,E)$.
\end{lemma}

\begin{proof}
Let $Z_{(-r,r)}$ be the double of $Z_r$ in $\tilde M$. Clearly, it suffices to consider the case when the support of $\tilde u$ is contained in $Z_{(-r,r)}$.
Since $\dom\tilde\DD_{\max}=\dom\tilde\DD_{\min}$, it suffices to show that if a sequence $\tilu_n\in C^\infty_c(Z_{(-r,r)},\tilE)$ converges to $\tilu$ in the graph norm of $\tilde\DD$ then $\tilu_n|_M$ converges in $H_\DD^1(M,E)$. This follows from the following estimate 
\begin{equation}\label{E:normtilD-D}
	\big\|\tilu|_M\|_{H_\DD^1(M,E)} \ \le \  \|\tilu\|_{\tilde\DD}, 
	\qquad \tilu\in C^\infty_c(Z_{(-r,r)},\tilE),
\end{equation}
which we  prove below. 

Since $\DD$ is a product on $Z_r$, we obtain from \eqref{E:productDD} that on $Z_{(-r,r)}$
\[
	\tilde\DD^*\tilde\DD\ = \ -\p_t^2+\AAA^2.
\]
Hence, on compactly supported sections $\tilu$ we have 
\[
	\big\|\tilde\DD\tilu\|^2_{L^2(\tilM,\tilE)} \ = \ 
	\big(\,\tilde\DD^*\tilde\DD\tilu,\tilu\,\big)_{L^2(\tilM,\tilE)} \ = \ 
	\|\p_t\tilu\|^2_{L^2(\tilM,\tilE)} \ + \ \|\AAA \tilu\|^2_{L^2(\tilM,\tilE)}.
\]
We conclude that 
\begin{multline}\notag
	\|\tilu\|^2_{\tilde\DD}\ := \ \|\tilu\|^2_{L^2(\tilM,\tilE)}
	\ + \  \|\tilde\DD\tilu\|^2_{L^2(\tilM,\tilE)} 
	\\ = \ \|\tilu\|^2_{L^2(\tilM,\tilE)} \ + \ 
	\|\p_t\tilu\|^2_{L^2(\tilM,\tilE)} \ + \ \|\AAA \tilu\|^2_{L^2(\tilM,\tilE)}
	\ \ge \ \big\|\tilu|_M\|_{H_\DD^1(M,E)}^2.
\end{multline}
\end{proof}

Let $\DD_c$ denote the operator $\DD$ with domain $C^\infty_c(M,E)$. Let $(\DD_{c})^\ad$ denote the adjoint of $\DD_c$ in the sense of functional analysis. Note that $(\DD_{c})^\ad\subset (\DD^*)_{\max}$, where, as usual, we  denote by $\DD^*$ the formal adjoint of $\DD$. 

Fix an arbitrary $u\in \dom(\DD_{c})^\ad$ and let $\tilu\in L^2(\tilM,\tilE)$ and $\tilv\in L^2(\tilM,\tilE)$ denote the sections whose restriction to $\tilM\backslash{}M$ are  equal to 0 and whose restriction to $M$ are equal to $u$ and $(\DD_c)^\ad u$ respectively. 

Let $\tilw\in C^\infty_c(\tilM,\tilE)$. The restriction of $w=\tilw|_M\in \dom\DD_c$. 
Since $\tilu|_{\tilM\backslash M}= \tilv|_{\tilM\backslash M}=0$ we obtain
\[
	(\tilde\DD\tilw,\tilu)_{L^2(\tilM,\tilE)} \ = \ (\DD_cw,u)_{L^2(M,E)}
	\ = \ \big(w,(\DD_c)^\ad u\big)_{L^2(M,E)} \ = \ (\tilw,\tilv)_{L^2(\tilM,\tilE)}.
\]
Hence, $\tilu$ is a weak solution of the equation $\tilde\DD^*\tilu= \tilv\in L^2(\tilM,\tilE)$. By elliptic regularity $\tilu\in H^1_{\loc}(\tilM,\tilE)$. It follows that $\tilu|_\pM= u|_\pM=0$. Also, by Lemma~\ref{L:ellipreg}, $u\in H^1_{\DD^*}(M,E)$. By Corollary~\ref{C:normequiv}, $u$ is in the domain of the minimal extension $(\DD^*)_{\min}$ of $(\DD^*)_{cc}$. Since $u$ is an arbitrary section in $\dom(\DD_{c})^\ad$, we conclude that $(\DD_c)^\ad\subset (\DD^*)_{\min}$. Hence the closure $\overline{\DD}_c$ of $\DD_c$ satisfies 
\[
	\overline{\DD}_c \ \subset \ \DD_{\max} \ = \ \big((\DD^*)_{\min}\big)^\ad
	 \ \subset \ \big((\DD_c)^\ad\big)^\ad\ =\ \overline{\DD}_c.
\]
Hence, $\overline{\DD}_c= \DD_{\max}$ as claimed in part (i)  of the theorem. 

(ii) \ By (i) $C_c^\infty(M,E)$ is dense in $\dom \DD_{\max}$. Hence, it follows from 
Lemma~\ref{L:restriction} that the extension exists and unique. To prove the  surjectivity recall that the space $H^{\rm fin}(\AAA)$, defined in \eqref{E:Hfin}, is dense in $\cHH(\AAA)$. Fix $\bfu\in \cHH(\AAA)$ and let $\bfu_i\to \bfu$ be a sequence of sections $\bfu_i\in H^{\rm fin}(\AAA)$ which converges to $\bfu$ in $\cHH(\AAA)$. Then, by Lemma~\ref{L:extension}, the sequence $\EE\bfu_i\in \dom \DD_{\max}$ is a Cauchy sequence and, hence, converges to an element $v\in \dom \DD_{\max}$. Then $\cRR v= \bfu$. 

(iii) \ The inclusion 
\[
	H_\DD^1(M,E)\ \subset\ 
	 \{u\in\dom\DD_{\max}:\,\cRR u\in H_\AAA^{1/2}(\dm,E_{\dm})\}
\]
follows directly from \eqref{E:HDM} and the Trace Theorem~\ref{T:tracethm}.

To show the opposite inclusion, choose $u\in \dom \DD_{\max}$ with $\cRR u\in H_\AAA^{1/2}(\pM,E_{\dm})$ and set $\bfv:=P^\AAA_{(0,\infty)}\cRR u$. Then 
\[
	u \ = \ \EE \bfv\ + \ (u-\EE \bfv).
\]
Using \eqref{E:extension2} we readily see that $\EE \bfv\in  H_\DD^1(M,E)$. Since $P^\AAA_{(0,\infty)}\cRR(u-\EE \bfv)= 0$ it follows from (i) and Lemma~\ref{L:normequiv}, that $u-\EE \bfv\in H_\DD^1(M,E)$. Thus $u\in H_\DD^1(M,E)$ as required. 

Finally, \eqref{E:genGreensfor} holds for $u,v\in C^\infty_c(M,E)$  by \eqref{E:Greensfor}. Since, by (i), $C^\infty_c(M,E)$ is dense in both $\dom \DD_{\max}$ and $\dom (\DD^*)_{\max}$, the equality for $u\in \dom\DD_{\max}$ and $v\in  \dom(\DD^*)_{\max}$ follows now from (i) and (ii) and Lemma~\ref{L:admap}.
\end{proof}

\section{Boundary value problems for strongly Callias-type operators}\label{S:bvp}

Moving on from last section, we study boundary value problems of a strongly Callias-type operator $\DD$ whose restriction to the boundary is $\AAA$. We introduce boundary conditions and elliptic boundary conditions for $\DD$ as certain closed subspaces of $\cHH(\AAA)$. In particular, we take a close look at an important elliptic boundary condition -- the Atiyah--Patodi--Singer boundary condition and obtain some results about it.

\subsection{Boundary conditions}\label{SS:bc}

Let $\DD$ be a strongly Callias-type operator. If $\dm=\emptyset$, then the minimal and maximal extensions of $\DD$ coincide, i.e., $\DD_{\min}=\DD_{\max}$. But when $\dm\ne\emptyset$ these two extensions are not equal. Indeed, the restrictions of elements of $\DD_{\min}$ to the boundary vanish identically by Corollary~\ref{C:normequiv}, while the restrictions of elements of $\DD_{\max}$ to the boundary form the whole space $\cHH(\AAA)$, cf. Theorem~\ref{T:maxdom}. The boundary value problems lead to closed extensions lying between $\DD_{\min}$ and $\DD_{\max}$.

\begin{definition}\label{D:bc}
A closed subspace $B\subset\cHH(\AAA)$ is called a \emph{boundary condition} for $\DD$. We will use the notations $\DD_{B,\max}$ and $\DD_B$ for the operators with the following domains
\[
\begin{aligned}
	\dom(\DD_{B,\max})&\;=\;\{u\in\dom\DD_{\max}:\cRR u\in B\},\\
	\dom\DD_B&\;=\;\{u\in H_\DD^1(M,E):\cRR u\in B\}\\
	&\;=\;
	\{u\in\dom\DD_{\max}:\cRR u\in B\cap H_\AAA^{1/2}(\dm,E_{\dm})\}.
\end{aligned}
\]
\end{definition}
We remark that if $B=\cHH(\AAA)$ then $\DD_{B,\max}= \DD_{\max}$. Also if $B=0$ then $\DD_{B,\max}= \DD_B= \DD_{\min}$.

By Theorem \ref{T:maxdom}.(ii), $\dom(\DD_{B,\max})$ is a closed subspace of $\dom\DD_{\max}$. Since the trace map extends to a bounded linear map $\cRR:H_\DD^1(M,E)\to H_\AAA^{1/2}(\dm,E_{\dm})$ and $H_\AAA^{1/2}(\dm,E_{\dm})\hookrightarrow\cHH(\AAA)$ is a continuous embedding, $\dom\DD_B$ is also a closed subspace of $H_\DD^1(M,E)$. We equip $\dom(\DD_{B,\max})$ with the graph norm of $\DD$ and $\dom\DD_B$ the $H_\DD^1$-norm.

In particular, $\DD_{B,\max}$ is a closed extension of $\DD$. Moreover, it follows immediately 
from Definition \ref{D:bc} that $B\subset H_\AAA^{1/2}(\dm,E_{\dm})$ if and only if $\DD_B=\DD_{B,\max}$. 
Thus in this case $\dom\DD_B=\dom\DD_{B,\max}$ is a complete Banach space with respect to both the $H^1_{\DD}$-norm and the graph norm. From \cite[p. 71]{ReSi1} we now obtain the following analogue of \cite[Lemma~7.3]{BaerBallmann12}:

\begin{lemma}\label{L:normequiv-bc}
Let $B$ be a boundary condition. Then $B\subset H_\AAA^{1/2}(\dm,E_{\dm})$ if and only if $\DD_B=\DD_{B,\max}$, and in this case the $H_\DD^1$-norm and graph norm of $\DD$ are equivalent on $\dom\DD_B$.
\end{lemma}

\subsection{Adjoint boundary conditions}\label{adbc}

For any boundary condition $B$, we have $\DD_{cc}\subset\DD_{B,\max}$. Hence the $L^2$-adjoint operators satisfy
\[
(\DD_{B,\max})^{\rm ad}\;\subset\;(\DD_{cc})^{\rm ad}\;=\;(\DD^*)_{\max}.
\]
From \eqref{E:genGreensfor}, we conclude that
\[
	\dom(\DD_{B,\max})^{\rm ad} \ = \ 
	\big\{\,v\in\dom(\DD^*)_{\max}:\,
	\big(c(\tau)\cRR u,\cRR v\big)=0\mbox{ for all }u\in\dom\DD_{B,\max}\,\big\}.
\]

By Theorem~\ref{T:maxdom}.(ii),  for any $\bfu\in B$ there exists $u\in\dom(\DD_{B,\max})$ with $\cRR u=\bfu$. Therefore
\[
(\DD_{B,\max})^{\rm ad}\;=\;(\DD^*)_{B^{\rm ad},\max}
\]
with
\begin{equation}\label{E:adbc}
	B^{\rm ad}\;:=\;
	\big\{\,\bfv\in\cHH(\AAA^\#):\,
		\big(c(\tau)\bfu,\bfv\big)=0\mbox{ for all }\bfu\in B\,\big\}.
\end{equation}
By Lemma \ref{L:admap}, $B^{\rm ad}$ is a closed subspace of $\cHH(\AAA^\#)$, thus is a boundary condition for $\DD^*$.

\begin{definition}\label{D:adbc}
The space $B^{\rm ad}$, defined by \eqref{E:adbc}, is called the \emph{adjoint} boundary condition to $B$.
\end{definition}

\subsection{Elliptic boundary conditions}\label{SS:ellbc}

We adopt the same definition of elliptic boundary conditions as in \cite{BaerBallmann12} for the case of non-compact boundary:

\begin{definition}\label{D:ellbc}
A boundary condition $B$ is said to be \emph{elliptic} if $B\subset H_\AAA^{1/2}(\dm,E_{\dm})$ and $B^{\rm ad}\subset H_{\AAA^\#}^{1/2}(\dm,E_{\dm})$.
\end{definition}

\begin{remark}
One can see from Lemma \ref{L:normequiv-bc} that when $B$ is an elliptic boundary condition, $\DD_{B,\max}=\DD_B$, $(\DD^*)_{B^{\rm ad},\max}=\DD_{B^{\rm ad}}^*$ and the two norms are equivalent. Definition \ref{D:ellbc} is also equivalent to saying that $\dom\DD_B\subset H_\DD^1(M,E)$ and $\dom\DD_{B^{\rm ad}}^*\subset H_{\DD^*}^1(M,E)$.
\end{remark}

The following two examples of elliptic boundary condition are the most important to our study (compare with  Examples~7.27, 7.28 of \cite[]{BaerBallmann12}).

\begin{example}[Generalized Atiyah--Patodi--Singer boundary conditions]\label{Ex:gAPS}
For any $a\in\RR$, let
\begin{equation}\label{E:gAPS}
B\;=\;B(a)\;:=\;H_{(-\infty,a)}^{1/2}(\AAA).
\end{equation}
This is a closed subspace of $\cHH(\AAA)$. In order to show that $B$ is an elliptic boundary condition, we only need to check that $B^{\rm ad}\subset H_{\AAA^\#}^{1/2}(\dm,E_{\dm})$. By Lemma \ref{L:admap}, $c(\tau)$ maps $H_{(-\infty,a)}^{1/2}(\AAA)$ to the subspace $H_{(-a,\infty)}^{1/2}(\AAA^\#)$ of $\hHH(\AAA^\#)$. Since there is a perfect pairing between $\cHH(\AAA^\#)$ and $\hHH(\AAA^\#)$, we see that
\begin{equation}\label{E:adgAPS}
B^{\rm ad}\;=\;H_{(-\infty,-a]}^{1/2}(\AAA^\#).
\end{equation}
Therefore $B$ is an elliptic boundary condition. It is called the \emph{the generalized Atiyah--Patodi--Singer boundary conditions} (or generalized APS boundary conditions for abbreviation). In particular, $B=H_{(-\infty,0)}^{1/2}(\AAA)$ will be called the \emph{Atiyah--Patodi--Singer boundary condition} and $B=H_{(-\infty,0]}^{1/2}(\AAA)$ will be called the \emph{dual Atiyah--Patodi--Singer boundary condition}. 
\end{example}

\begin{remark}\label{R:aps-dual aps}
One can see from \eqref{E:adgAPS} that the adjoint of the APS boundary condition for $\DD$ is the dual APS boundary condition for $\DD^*$.
\end{remark}

\begin{example}[Transmission conditions]\label{Ex:transmission}
Let $M$ be a complete manifold. For simplicity, first assume that $\dm=\emptyset$. Let $N\subset M$ be a hypersurface such that cutting $M$ along $N$ we obtain a manifold $M'$ (connected or not) with two copies of $N$ as boundary. So we can write $M'=(M\setminus N)\sqcup N_1\sqcup N_2$.

Let $E\to M$ be a Dirac bundle over $M$ and $\DD:C^\infty(M,E)\to C^\infty(M,E)$ be a strongly Callias-type operator. They induce Dirac bundle $E'\to M'$ and strongly Callias-type operator $\DD':C^\infty(M',E')\to C^\infty(M',E')$ on $M'$. We assume that all structures are product near $N_1$ and $N_2$. Let $\AAA$ be the restriction of $\DD'$ to $N_1$. Then $-\AAA$ is the restriction of $\DD'$ to $N_2$ and, thus, the restriction of $\DD'$ to $\dm'$ is $\AAA'=\AAA\oplus-\AAA$. 

For $u\in H_\DD^1(M,E)$ one gets $u'\in H_{\DD'}^1(M',E')$ (by Lemma \ref{L:ellipreg}) such that $u'|_{N_1}=u'|_{N_2}$. We use this as a boundary condition for $\DD'$ on $M'$ and set
\begin{equation}\label{E:transmission}
		B \;:=\;
	\big\{(\bfu,\bfu)\in H_\AAA^{1/2}(N_1,E_{N_1})\oplus H_{-\AAA}^{1/2}(N_2,E_{N_2})\big\}.
\end{equation}

\begin{lemma}\label{L:transmission}
The subspace \eqref{E:transmission} is an elliptic boundary condition, called the {\em transmission boundary condition}.
\end{lemma}

\begin{proof}

First we show that $B$ is a boundary condition, i.e. is a closed subspace of $\cHH(\AAA')$. Clearly $B$ is a closed subspace of $H_{\AAA'}^{1/2}(\dm',E_{\dm'})$. Thus it suffices to show that the $H_{\AAA'}^{1/2}$-norm and $\cHH(\AAA')$-norm are equivalent on $B$. Since any two norms are equivalent on the finite-dimensional eigenspace of $\AAA'$ associated to eigenvalue 0, we may assume that 0 is not in the spectrum of $\AAA'$. Write
\[
	\bfu \;=\; 
	P_{(-\infty,0)}^\AAA\bfu+P_{(0,\infty)}^\AAA\bfu
	\;=:\; \bfu^-+\bfu^+.
\]
Notice that $P_I^{\AAA'}=P_I^\AAA\oplus P_I^{-\AAA}=P_I^\AAA\oplus P_{-I}^\AAA$ for any subset $I\subset\RR$. We have
\[
P_{(-\infty,0)}^{\AAA'}(\bfu,\bfu)\;=\;(\bfu^-,\bfu^+),\quad P_{(0,\infty)}^{\AAA'}(\bfu,\bfu)\;=\;(\bfu^+,\bfu^-).
\]
Notice also that 
\[
	\|\bfu^+\|_{H_{\AAA}^{\pm1/2}(N_1)}\ = \ \|\bfu^+\|_{H_{-\AAA}^{\pm1/2}(N_2)}
\]
and similar equality holds for $\bfu^-$. It follows that 
\[
	\|(\bfu,\bfu)\|_{H_{\AAA'}^{\pm1/2}(\dm')}^2	 \ = \ 
	2\,\|\bfu\|_{H_\AAA^{\pm1/2}(N_1)}^2
	\ = \ 2\,\|\bfu\|_{H_{-\AAA}^{\pm1/2}(N_2)}^2.
\]
Using the above equations we get
\[
\begin{aligned}
\|(\bfu,\bfu)\|_{\cHH(\AAA')}^2&\;=\;\|(\bfu^-,\bfu^+)\|_{H_{\AAA'}^{1/2}(\dm')}^2+\|(\bfu^+,\bfu^-)\|_{H_{\AAA'}^{-1/2}(\dm')}^2\\
&\;=\;\|\bfu^-\|_{H_\AAA^{1/2}(N_1)}^2+\|\bfu^+\|_{H_{-\AAA}^{1/2}(N_2)}^2+\|\bfu^+\|_{H_\AAA^{-1/2}(N_1)}^2+\|\bfu^-\|_{H_{-\AAA}^{-1/2}(N_2)}^2\\
&\;=\;\|\bfu\|_{H_\AAA^{1/2}(N_1)}^2+\|\bfu\|_{H_\AAA^{-1/2}(N_1)}^2\\
&\;=\;\frac{1}{2}\Big(\|(\bfu,\bfu)\|_{H_{\AAA'}^{1/2}(\dm')}^2+\|(\bfu,\bfu)\|_{H_{\AAA'}^{-1/2}(\dm')}^2\Big)\\
&\;\ge\;\frac{1}{2}\|(\bfu,\bfu)\|_{H_{\AAA'}^{1/2}(\dm')}^2.
\end{aligned}
\]
The other direction of inequality is trivial. So $B$ is also closed in $\cHH(\AAA')$ and hence is a boundary condition.

In order to show that $B$ is an elliptic boundary condition we need to prove that $B^{\rm ad}\subset\;H_{\AAA^{\prime\#}}^{1/2}(\dm',E_{\dm'})$. Let $\tau'$ be the inward unit normal vector to the boundary of $M'$. Then $\tau'=\tau\oplus(-\tau)$, where $\tau$ is the inward unit normal vector to the boundary component $N_1$. It is easy to see that
\[
	B^{\rm ad}\;=\;
	\big\{\,(\bfv,\bfv) \ \in \
	   H_{\AAA^\#}^{-1/2}(N_1,E_{N_1})\oplus H_{-\AAA^\#}^{-1/2}(N_2,E_{N_2})\,\big\}
	\;\cap\;\cHH(\AAA^{\prime\#}).
\]
Again by decomposing $\bfv$ in terms of $\bfv^-$ and $\bfv^+$ like above, one can get that $\bfv\in H_{\AAA^\#}^{1/2}(N_1,E_{N_1})$. Therefore
\begin{equation}\label{E:adtransmission}
	B^{\rm ad}\;=\;
	\big\{\,(\bfv,\bfv)\in H_{\AAA^\#}^{1/2}(N_1,E_{N_1})\oplus 
		H_{-\AAA^\#}^{1/2}(N_2,E_{N_2})\,\big\}
		\;\subset\;H_{\AAA^{'\#}}^{1/2}(\dm',E_{\dm'}).
\end{equation}
Therefore $B$ is an elliptic boundary condition. In addition, if $\DD=\DD^*$, then $B=B^{\rm ad}$ and $\DD'_B$ is a self-adjoint operator.
\end{proof}

If $M$ has nonempty boundary and $N$ is disjoint from $\dm$, we assume that an elliptic boundary condition is posed for $\dm$. Then one can apply the same arguments as above to pose the transmission condition for $N_1\sqcup N_2$ and keep the original condition for $\dm$.
\end{example}

\section{Index theory for strongly Callias-type operators}\label{S:index}

In this section we show that an elliptic boundary value problem for a strongly Callias-type operator is Fredholm. As two typical examples, the indexes of APS  and transmission boundary value problems are interesting and are used to prove the splitting theorem, which allows to compute the index by cutting and pasting.

\subsection{Fredholmness}\label{SS:fred}

Let $\DD:C^\infty(M,E)\to C^\infty(M,E)$ be a strongly Callias-type operator. The growth assumption of the Callias potential guarantees that $\DD$ is \emph{invertible at infinity}.

\begin{lemma}\label{L:invatinf}
A strongly Callias-type operator $\DD:C^\infty(M,E)\to C^\infty(M,E)$ is \emph{invertible at infinity} (or \emph{coercive at infinity}). Namely, there exist a constant $C>0$ and a compact set $K\subset M$ such that
\begin{equation}\label{E:invatinf}
	\|\DD u\|_{L^2(M,E)}\;\ge\;C\,\|u\|_{L^2(M,E)},
\end{equation}
for all $u\in C_{cc}^\infty(M,E)$ with ${\rm supp}(u)\cap K=\emptyset$.
\end{lemma}

\begin{remark}\label{R:invatinf}
Note that this property is independent of the boundary condition of $\DD$.
\end{remark}

\begin{proof}
By Definition \ref{D:strCallias}, for a fixed $R>0$, one can find an $R$-essential support $K_R\subset M$ for $\DD$, so that
\begin{multline}\notag
	\|\DD u\|_{L^2(M,E)}^2\;=\;
	(\DD u,\DD u)_{L^2(M,E)}\; = \; (\DD^*\DD u,u)_{L^2(M,E)} \\
	\;=\;(D^2u,u)_{L^2(M,E)}\;+\;\big((\Phi^2+i[D,\Phi])u,u\big)_{L^2(M,E)} \\
	\;\ge\;
	\|Du\|_{L^2(M,E)}^2\;+\;R\,\|u\|_{L^2(M,E)}^2 \;\ge\;R\,\|u\|_{L^2(M,E)}^2
\end{multline}
for all $u\in C_{cc}^\infty(M,E)$ with support outside $K_R$.
\end{proof}

Recall that, for $\dm=\emptyset$, a first-order essentially self-adjoint elliptic operator which is invertible at infinity is Fredholm (cf. \cite[Theorem 2.1]{Anghel93}). If $\dm\ne\emptyset$ is compact, an analogous result (with elliptic boundary condition) is proven in \cite[Theorem 8.5, Corollary 8.6]{BaerBallmann12}. We now generalize the result of \cite{BaerBallmann12} to the case of non-compact boundary

\begin{theorem}\label{T:Fred}
Let $\DD_B:\dom\DD_B\to L^2(M,E)$ be a strongly Callias-type operator with elliptic boundary condition. Then $\DD_B$ is a Fredholm operator.
\end{theorem}

\begin{proof}
A bounded linear operator $T:X\to Y$ between two Banach spaces has finite-dimensional kernel and closed image if and only if every bounded sequence $\{x_n\}$ in $X$ such that $\{Tx_n\}$ converges in $Y$ has a convergent subsequence in $X$, cf. \cite[Proposition 19.1.3]{Hormander}. We show below that both, $\DD:\dom\DD_B\to L^2(M,E)$ and $(\DD^*)_{B^\ad}:\dom(\DD^*)_{B^\ad}\to L^2(M,E)$ satisfy this property. 

We let $\{u_n\}$ be a bounded sequence in $\dom\DD_B$ such that $\DD u_n\to v\in L^2(M,E)$ and want to show that $\{u_n\}$ has a convergent subsequence in $\dom\DD_B$.

Recall that we assume that there is a neighborhood $Z_r= [0,r)\times\pM\subset M$ of the boundary such that the restriction of $\DD$ to $Z_r$ is product. For $(t,y)\in Z_r$ we set $\chi_1(t,y)= \chi(t)$ where $\chi$ is the cut-off function defined in \eqref{E:cut-off}. We set $\chi_1(x)\equiv0$ for $x\not\in Z_r$. Then $\chi_1$ is supported on $Z_{2r/3}$ and identically equal to 1 on $Z_{r/3}$. We also note that $d\chi_1$ is uniformly bounded and supported in $Z_{2r/3}$. 

Let the compact set $K\subset M$ and a constant $C>0$ be as in Lemma \ref{L:invatinf}. We choose two more  smooth cut-off functions $\chi_2,\chi_3:M\to[0,1]$ such that
\begin{itemize}
\item $K':={\rm supp}(\chi_2)$ is compact and $\chi_1+\chi_2\equiv1$ on $K$;
\item $\chi_1+\chi_2+\chi_3\equiv1$ on $M$.
\end{itemize}
As a consequence, $d\chi_3$ is uniformly bounded and ${\rm supp}(d\chi_3)\subset Z_{2r/3}\cup K'$. We denote 
\begin{equation}\label{E:C'=sup}
	\kappa\ = \ \sup \, |d\chi_3|.
\end{equation}

Lemma~\ref{L:comHDtoL2} and the classical Rellich Embedding Theorem imply that, passing to a subsequence, we can assume that the restrictions of $u_n$ to $Z_{2r/3}$ and to $K'$ are $L^2$-convergent. Then in the inequality
\begin{multline}\label{E:estimate chi12}
	\|u_n-u_m\|_{L^2(M,E)}\\
	\le\;
	\|\chi_1(u_n-u_m)\|_{L^2(M,E)} \ + \ \|\chi_2(u_n-u_m)\|_{L^2(M,E)}
	\ + \ \|\chi_3(u_n-u_m)\|_{L^2(M,E)} 
	\\ \le \ 
	\|u_n-u_m\|_{L^2(Z_{2r/3},E)} \ + \ \|u_n-u_m\|_{L^2(K',E)}
	\ + \ \|\chi_3(u_n-u_m)\|_{L^2(M,E)} 
\end{multline}
the first two terms on the right hand side converge to 0 as $n,m\to\infty$. To show that $\{u_n\}$ is a Cauchy sequence it remains to prove that the last summand converges to 0 as well. We use  Lemma \ref{L:invatinf} to get
\begin{multline}\label{E:estimate chi3}\notag
	\|\chi_3(u_n-u_m)\|_{L^2(M,E)}
	\;\le\;
	 C^{-1}\,\|\DD\chi_3(u_n-u_m)\|_{L^2(M,E)} 
	\\ \le\;
	C^{-1}\,\|c(d\chi_3)(u_n-u_m)\|_{L^2(M,E)} 
	+ \ C^{-1}\|\chi_3(\DD u_n-\DD u_m)\|_{L^2(M,E)} \\
	\;\le\;
	\kappa\, C^{-1}\,\big(\,\|u_n-u_m\|_{L^2(Z_{2r/3},E)} + \|u_n-u_m\|_{L^2(K',E)}\,\big)
	\ + \  C^{-1}\,\|\DD u_n-\DD u_m\|_{L^2(M,E)},
\end{multline}
where in the last inequality we used \eqref{E:C'=sup}. Since $\DD u_n$, $u_n|_{Z_{2r/3}}$ and $u_n|_{K'}$ are all convergent,  $\chi_3(u_n-u_m)\to 0$ in $L^2(M,E)$ as $m,n\to \infty$. Combining with \eqref{E:estimate chi12} we conclude that $\{u_n\}$ is a Cauchy sequence and, hence,   converges in $L^2(M,E)$.

Now both $\{u_n\}$ and $\{\DD u_n\}$ are convergent in $L^2(M,E)$. Hence $\{u_n\}$ converges in the graph norm of $\DD$. Since $B$ is an elliptic boundary condition, by Lemma \ref{L:normequiv-bc}, the $H_\DD^1$-norm and graph norm of $\DD$ are equivalent on $\dom\DD_B$. So we proved that $\{u_n\}$ is convergent in $\dom\DD_B$. Therefore $\DD_B$ has finite-dimensional kernel and closed image. Since $\DD^*$ is also a strongly Callias-type operator, exactly the same arguments apply to $(\DD^*)_{B^{\rm ad}}$ and we get that $\DD_B$ is Fredholm.
\end{proof}

\begin{definition}\label{D:index}
Let $\DD$ be a strongly Callias-type operator on a complete Riemannian manifold $M$ which is  product near the boundary. Let $B\subset H^{1/2}_{\AAA}(\pM,E_\pM)$ be an elliptic boundary condition for $\DD$. The integer 
\begin{equation}\label{E:def of index}
		\ind\DD_B\;:=\;\dim\ker\DD_B-\dim\ker(\DD^*)_{B^{\rm ad}}\;\in\;\ZZ
\end{equation}
is called the {\em index of the boundary value problem} $\DD_B$. 
\end{definition}

It follows directly from \eqref{E:def of index} that 
\begin{equation}\label{E:indD-indD*}
	\ind\, (\DD^*)_{B^{\rm ad}}\ = \ -\ind \DD_B.
\end{equation}

\subsection{Dependence of the index on the boundary conditions}\label{SS:dep on bndry}

We say that two closed subspaces $X_1$, $X_2$ of a Hilbert space $H$ are {\em finite rank perturbations}  of each other if there exists a finite dimensional subspace $Y\subset H$ such that $X_2\subset X_1\oplus Y$ and the quotient space $(X_1\oplus Y)/X_2$ has finite dimension. We define the {\em relative index} of $X_1$ and $X_2$ by 
\begin{equation}\label{E:rel index}
	[X_1,X_2] \ := \ \dim\, (X_1\oplus Y)/X_2\ - \ \dim Y.
\end{equation}
One easily sees that the relative index is independent of the choice of $Y$. We also note that if $X_1$ and $X_2$ are finite rank perturbations of each other, then  $X_1$ and the orthogonal complement $X_2^\perp$ of $X_2$ form a Fredholm pair in the sense of \cite[\S IV.4.1]{Kato95book} and the relative index $[X_1,X_2]$ is equal to extension of $M$ by a cylinder

The following lemma follows immediately from the definition of the relative index.

\begin{lemma}\label{L:relindex sign}
$\displaystyle [X_2,X_1]\ = \ [X_1^\perp,X_2^\perp] \ = \ -\,[X_1,X_2]$.
\end{lemma}

\begin{proposition}\label{P:change of B}
Let $\DD$ be a strongly Callias-type operator on $M$ and let $B_1$ and $B_2$ be elliptic boundary conditions for $\DD$. If\/ $B_1, B_2\in H_\AAA^{1/2}(\pM,E_\pM)$ are finite rank perturbations of each other, then 
\begin{equation}\label{E:change of B}
	\ind \DD_{B_1}\ - \ \ind \DD_{B_2}\ = \ [B_1,B_2].
\end{equation}
\end{proposition}
The proof of the proposition is a verbatim repetition of the proof of Theorem~8.14 of \cite{BaerBallmann12}.

As an immediate consequence of Proposition~\ref{P:change of B} we obtain the following
\begin{corollary}\label{C:change of B}
Let $\AAA$ be the restriction of $\DD$ to $\pM$ and let  $B_0= H^{1/2}_{(-\infty,0)}(\AAA)$ and $B_1= H^{1/2}_{(-\infty,0]}(\AAA)$ be the APS and the dual APS boundary conditions respectively, cf. Example~\ref{Ex:gAPS}. Then 
\begin{equation}\label{E:change of B2}
	\ind \DD_{B_1} \ = \ \ind\DD_{B_0}\ + \ \dim\ker \AAA.
\end{equation}
More generally, let $B(a)=H_{(-\infty,a)}^{1/2}(\AAA)$ and $B(b)=H_{(-\infty,b)}^{1/2}(\AAA)$ be two generalized APS boundary conditions with $a<b$. Then
\[
	\ind \DD_{B(b)} \ = \ \ind\DD_{B(a)}\ + \ \dim L_{[a,b)}^2(\AAA).
\]
\end{corollary}
\subsection{The splitting theorem}\label{SS:splittingthm}

We use the notation of Example \ref{Ex:transmission}.

\begin{theorem}\label{T:splitting}
Suppose $M,\DD,M',\DD'$ are as in Example \ref{Ex:transmission}. Let $B_0$ be an elliptic boundary condition on $\dm$. Let $B_1=H_{(-\infty,0)}^{1/2}(\AAA)$ and $B_2=H_{[0,\infty)}^{1/2}(\AAA)=H_{(-\infty,0]}^{1/2}(-\AAA)$ be the APS and the dual APS boundary conditions along $N_1$ and $N_2$, respectively. Then $\DD'_{B_0\oplus B_1\oplus B_2}$ is a Fredholm operator and
\[
	\ind\DD_{B_0}\;=\;\ind\DD'_{B_0\oplus B_1\oplus B_2}.
\]
\end{theorem}

\begin{proof}
We assume that $\dm=\emptyset$. The proof of the general case is exactly the same, but the notation is more cumbersome.  Since $B_1\oplus B_2$ is an elliptic boundary condition for $\DD'$, the boundary value problem $\DD'_{B_1\oplus B_2}$ is Fredholm.  We need to show the index identity, which now is
\begin{equation}\label{E:splitting-1}
	\ind\DD\;=\;\ind\DD'_{B_1\oplus B_2}.
\end{equation}

Let $B$ denote the transmission condition on $\dm'$. Then, using the canonical pull-back of sections from $E$ to $E'$, we have
\[
\dom\DD\;=\;\{u\in H_{\DD'}^1(M',E'):\cRR u\in B\}\;=\;\dom\DD'_B
\]
and
\begin{equation}\label{E:splitting-2}
\ind\DD\;=\;\ind\DD'_B.
\end{equation}

We now proceed as in the proof of Theorem~8.17 of \cite{BaerBallmann12} with minor changes. The main idea is to construct a deformation of the transmission boundary condition $B$ into the APS boundary condition $B_1\oplus{}B_2$ and thus to show that $\ind \DD'_B= \ind \DD'_{B_1\oplus B_2}$.

Recall that in Example \ref{Ex:transmission}, we express any element $(\bfu,\bfu)$ of $B$ as $(\bfu^-+\bfu^+,\bfu^++\bfu^-)$, where $\bfu^-=P_{(-\infty,0)}^\AAA\bfu$ and $\bfu^+=P_{[0,\infty)}^\AAA\bfu$. Note that $\bfu^-\in B_1$ and $\bfu^+\in B_2$. For $0\le s\le1$, we define a family of boundary conditions
\[
  \begin{aligned}
	B_{1,s}\ &:= \ 
	\big\{\, \bfu^-+(1-s)\bfu^+:\, \bfu\in H^{1/2}_\AAA(N_1,E_{N_1})\,\big\};\\
	B_{2,s}\ &:= \ 
	\big\{\, \bfu^++(1-s)\bfu^-:\, \bfu\in H^{1/2}_{-\AAA}(N_2,E_{N_2})\simeq 
			H^{1/2}_\AAA(N_1,E_{N_1}) \,\big\},
 \end{aligned}
\]
and a family of isomorphisms
\[
	k_s:B\to B_{1,s}\oplus B_{2,s},\qquad 
	k_s(\bfu,\bfu):=(\bfu^-+(1-s)\bfu^+,\bfu^++(1-s)\bfu^-).
\]
Here $k_0=\id$ and $k_1$ is an isomorphism from $B$ to $B_1\oplus B_2$. One can follow the arguments of   Lemma \ref{L:transmission} to check that for each $s\in[0,1]$,
\[
B_{1,s}^{\rm ad}\oplus B_{2,s}^{\rm ad}\;=\;\big\{(\bfv^-+(1-s)\bfv^+,\bfv^++(1-s)\bfv^-)\in H_{\AAA^\#}^{1/2}(N_1,E_{N_1})\oplus H_{-\AAA^\#}^{1/2}(N_2,E_{N_2})\big\},
\]
where $\bfv^-\in H_{(-\infty,0]}^{1/2}(\AAA^\#)$ and $\bfv^+\in H_{(0,\infty)}^{1/2}(\AAA^\#)$. Thus $B_{1,s}\oplus B_{2,s}$ is an elliptic boundary condition for all $s\in[0,1]$ and we get a family of Fredholm operators $\{\DD'_{B_{1,s}\oplus B_{2,s}}\}_{0\le s\le1}$.

By definition,
\[
	(k_{s_1}-k_{s_2})(\bfu,\bfu)\;=\;(s_2-s_1)(\bfu^+,\bfu^-).
\]
Notice that $\|(\bfu^+,\bfu^-)\|_{H_{\AAA'}^{1/2}(\dm',E')}\le\|(\bfu,\bfu)\|_{H_{\AAA'}^{1/2}(\dm',E')}$. Hence, for $s_1,s_2\in[0,1]$ with $|s_1-s_2|<\varepsilon$,  the operator
\[
	k_{s_1}-k_{s_2}:\, B \to \ H_{\AAA'}^{1/2}(\dm',E')
\]
has a norm not greater than $\varepsilon$. This implies that $\{k_s\}$ is a continuous family of isomorphisms from $B$ to $H_{\AAA'}^{1/2}(\dm',E')$. The following steps are basically from \cite[Lemma 8.11, Theorem 8.12]{BaerBallmann12}. Roughly speaking, one can construct a continuous family of isomorphisms
\[
K_s\;:\;\dom\DD'_B\;\to\;\dom\DD'_{B_{1,s}\oplus B_{2,s}}.
\]
Then by composing $\DD'_{B_{1,s}\oplus B_{2,s}}$ and $K_s$, one gets a continuous family of Fredholm operators on the fixed domain $\dom\DD'_B$. The index is constant. Since $K_1$ is an isomorphism, we have
\begin{equation}\label{E:splitting-3}
\ind\DD'_B\;=\;\ind\DD'_{B_1\oplus B_2}.
\end{equation}
At last, \eqref{E:splitting-1} follows from \eqref{E:splitting-2} and \eqref{E:splitting-3}. This completes the proof.
\end{proof}

\subsection{A vanishing theorem}\label{SS:vanishing}
As a first application of the splitting theorem~\ref{T:splitting} we prove the following vanishing result.

\begin{corollary}\label{C:vanishing}
Suppose that there exists $R>0$ such that $\DD$ has an empty $R$-essential support.  Let $B_0=H_{(-\infty,0)}^{1/2}(\AAA)$ be the APS boundary condition, cf. Example~\ref{Ex:gAPS}. Then 
\begin{equation}\label{E:vanishing}
	\ind \DD_{B_0}\ = \ 0.
\end{equation}
\end{corollary}

\begin{proof}
Since all our structures are product near $\pM$ and the $R$-essential support of $\DD$ is empty, the $R$-essential support of the restriction $\AAA:=A-ic(\tau)\Phi$\/ of\/ $\DD$ to $\pM$ is also empty. In particular, the operator $\AAA^2$ is strictly positive. It follows that 0 is not in the spectrum of $\AAA$. 

First consider the case when $M=[0,\infty)\times N$ is a cylinder and \eqref{E:productDD} holds everywhere on $M$.  In particular, this means that $\Phi(t,y)= \Phi(0,y)$ for all $t\in [0,\infty)$ and all $y\in N=\pM$. To distinguish this case from the general case, we denote the Callias-type operator on the cylinder by $\DD'$. Any $u\in \dom(\DD_{B_0}')$ can be written as 
\[
	u\ = \ \sum_{j=1}^\infty\, a_j(t)\,\bfu_j,
\]
where $\bfu_j$ is a unit eigensection of $\AAA$ with eigenvalue $\lambda_j<0$. If $\DD' u=0$ then $a_j(t)=c_je^{-\lambda_j t}$ for all $j$. It follows that $u\not\in L^2(M,E)$. In other words, there are no $L^2$-sections in the kernel of $\DD'_{B_0}$. Similarly, one proves that the kernel of $(\DD^{\prime*})_{B_0^{\rm ad}}$ is trivial. Thus 
\begin{equation}\label{E:indD'}
	\ind \DD'_{B_0}\ = \ 0.
\end{equation}

Let us return to the case of a general manifold $M$. Let 
\[
	\tilde{M}   \ := \ \big(\, (-\infty,0]\times\pM\,\big)\cup_\pM M
\]
be the extension of $M$ by a cylinder. Then $\tilde{M}$ is a complete manifold without boundary. Since all our structures are product near $\pM$ they extend naturally to $\tilde{M}$. Let $\tilde{\DD}$ be the induced strongly Callias-type operator on $\tilde{M}$. It has an empty $R$-essential support. Hence, $\tilde{\DD}^*\tilde{\DD}>0$ and $\tilde{\DD}\tilde{\DD}^*>0$. It follows that 
\begin{equation}\label{E:indtildeD}
	\ind \tilde{\DD}=0.
\end{equation} 
Notice that the restriction of $\tilde{\DD}$ to the cylinder is the operator $\DD'$ whose $R$-essential support is empty and whose restriction to the boundary is $-\AAA$.
Let  
\[
	B_0'\ :=\  H^{1/2}_{(-\infty,0)}(-\AAA) \ = \ H^{1/2}_{(0,\infty)}(\AAA)
\] 
denote the APS boundary condition for $\DD'$. Since $\AAA$ is invertible, $B_0'$ coincides with the dual APS boundary condition for $\DD'$. Hence, by  the splitting theorem~\ref{T:splitting} 
\begin{equation}\label{E:spliting vanishing}
	\ind \tilde{\DD}\ = \ \ind \DD_{B_0}\ + \ \ind \DD'_{B_0'}.
\end{equation}
The second summand on the right hand side of \eqref{E:spliting vanishing} vanishes by \eqref{E:indD'}. The corollary follows now from \eqref{E:indtildeD}.
\end{proof}

\section{Reduction to an essentially cylindrical manifold}\label{S:esscylindrical}

In this section we reduce the computation of the index  of an APS boundary value problem to a computation on a simpler manifold which we call {\em essentially cylindrical}. 

\begin{definition}\label{D:esscylindrical}
An {\em essentially cylindrical} manifold  $M$ is a complete Riemannian manifold whose boundary is a disjoint union of two components, $\pM=N_0\sqcup N_1$, such that 
\begin{enumerate}
\item
there exist a compact set $K\subset M$, an open manifold $N$, and an isometry $M\backslash K\simeq [0,\varepsilon]\times N$;

\item
under the above isometry $N_0\backslash K =\{0\}\times N$ and $N_1\backslash K =\{\varepsilon\}\times N$. 
\end{enumerate}
\end{definition}

\begin{remark}\label{R:esscylindrical}
Essentially cylindrical manifolds should not be confused with  manifolds with cylindrical ends. In a manifold $M$ with cylindrical ends there is a compact set $K$ such that $M\backslash K= [0,\infty)\times N$ is a cylinder with infinite axis $[0,\infty)$ and compact base $N$. As opposed to it, in an essentially cylindrical manifold, $M\backslash K$ is a cylinder with compact axis $[0,\varepsilon]$ and non-compact base $N$. 
\end{remark}

\subsection{Almost compact essential support}\label{SS:almost compact support}

We now return to the setting of Section~\ref{S:strCallias&maxdom}. In particular,  $M$ is a complete Riemannian manifold with non-compact boundary $\pM$ and there is a fixed isometry between a neighborhood of $\pM$ and the product $Z_r=[0,r)\times\pM$, cf. \eqref{E:Zr};\/ $\DD=D+i\Phi$ is a strongly Callias-type operator (cf. Definition~\ref{D:strCallias}) whose restriction to $Z_r$ is a product \eqref{E:productDD}.

\begin{definition}\label{D:almost compact support}
An {\em almost compact essential support of\, $\DD$} is a smooth submanifold $M_1\subset M$  with smooth boundary, which contains $\pM$ and such that 
\begin{enumerate}
\item 
$M_1$ contains an essential support for $\DD$, cf. Definition~\ref{D:strCallias};

\item 
there exist a compact set $K\subset M$ and $\varepsilon\in (0,r)$ such that 
\begin{equation}\label{E:almost compact}
	M_1\backslash K \ = \ (\pM\backslash K)\times [0,\varepsilon] 
	\subset Z_r. 
\end{equation}
\end{enumerate}
\end{definition}

Note that any almost compact essential support is  an essentially cylindrical manifold, one component of whose boundary is $\pM$ and $\AAA$ has an empty essential support on the other component of the boundary. Also the restriction of $\DD$ to the subset \eqref{E:almost compact} is given by \eqref{E:productDD}.

\begin{lemma}\label{L:almost compact support}
For every strongly Callias-type operator which is product on $Z_r$ there exists an almost compact essential support.
\end{lemma}
\begin{proof}
Fix $R>0$ and let  $K_R$ be a compact essential support for $\DD$. The union 
\[
	M'\ := \ \big([0,r/2]\times \pM\big)\cup K_R
\]
satisfies all the properties of an almost compact essential support, except that its boundary is not necessarily smooth. 

To find a similar set with a smooth boundary, let us fix an open cover of $M$ by three open sets $M= V_1\cup V_2\cup V_3$ where $V_1\supset K_R$ has compact closure, $V_2\subset  ([0,r)\times \p M)\setminus K_R$, and $V_3\subset M\setminus\big(([0,5r/6)\times \p M)\cup K_R\big)$.
Fix a partition of unity $\{\phi_1, \phi_2,\phi_3\}$ ($\phi_1+\phi_2+\phi_3\equiv 1$)  subordinate to this cover, so that $\supp \phi_j\in V_j$.    Let  $\rho:M\to [0,\infty)$ be a smooth function on $M$ such that 
\[
	\rho(x) \,-\,   {\rm dist}(x,M') \,<\, \frac{r}6,
	 \qquad \text{for all } x\in M.
\]
Let $t:Z_r\simeq [0,r)\times \p M\to [0,r)$ be the natural  projection. 

Without loss of generality we can assume that $r<1$. Define a smooth function on $M$ by 
\[
	f \ :=\ \rho\phi_1\ + \ t\phi_2\ + \ \phi_3
\]
Let $\delta\in (\frac{2r}3,\frac{5r}6)$ be a regular value of $f$. 
\[
	M_\delta\ : = \ f^{-1}([0,\delta])
\]
contains $M'$, has smooth boundary $\p M_\delta= f^{-1}(\delta)$,  and there exists a compact set $K'\supset K_R$ such that $M_\delta =  \big([0,\delta]\times \pM\big)\cup K'$. Hence, $M_\delta$ is an almost compact essential support for $\DD$.
\end{proof}

\subsection{The index on an almost compact essential support}\label{SS:index ac}
Suppose $M_1\subset M$ is an almost compact essential support for $\DD$ and let $N_1\subset M$ be such that $\pM_1=\pM\sqcup N_1$. The restriction of $\DD$ to a neighborhood of $N_1$ need not be product. Since in this paper we only consider boundary value problems for operators which are product near the boundary, we first deform $\DD$ to a product form. Note that if $K$ is as in Definition~\ref{D:almost compact support} then $\DD$ is product in a neighborhood of $N_1\backslash K$. It follows that we only need to deform $\DD$ in a relatively compact neighborhood of $N_1\cap K$. 
More precisely let $\varepsilon$ be as in \eqref{E:almost compact}. We choose $\delta\in(0,\varepsilon)$ and  a tubular neighborhood $U\subset M$ of $N_1$ such that 
\begin{equation}\label{E:productN1}
		U\backslash K\ = \ 
	(\varepsilon-\delta,\varepsilon+\delta)\times (N_1\backslash K)
	\subset Z_r.
\end{equation}
We now identify $U$ with the product $(\varepsilon-\delta,\varepsilon+\delta)\times N_1$ in a way compatible with \eqref{E:productN1}. The next lemma shows that one can find a strongly Callias-type operator $\DD'$ which is a product near $N_1$ and  differs from $\DD$ only on a compact set.

\begin{definition}\label{D:compact perturbation}
Fix a new Riemannian metric on $M$ and a new Hermitian metric on $E$ which differ from the original metrics only on a compact set $K'\subset M$. Let $c':T^*M\to \End(E)$ and let $\n^{E'}$ be a Clifford multiplication and a Clifford connection compatible with the new metrics, which also differ from $c$ and $\n^E$ only on $K'$. Let $D'$ be the Dirac operator defined by $c'$ and $\n^{E'}$. Finally, let $\Phi'\in \End(E)$ be a new Callias potential which is equal to $\Phi$ on $M\backslash{K'}$.  In this situation we say that the operator $\DD':= D'+i\Phi'$ is a {\em compact perturbation} of $\DD$. 
\end{definition}

Clearly, if $\DD'$ is a compact perturbation of $\DD$ which is equal to $\DD$ near $\pM$,  then every elliptic  boundary condition $B$ for $\DD$ is also elliptic for $\DD'$. Then  the stability of the index implies that
\begin{equation}\label{E:D=D'}
	\ind \DD_B \ = \ \ind \DD'_B.
\end{equation}

\begin{lemma}\label{L:compact perturbation}
In the situation of Subsection~\ref{SS:index ac} there exists a compact perturbation $\DD'$ of $\DD$ which is product near $\pM_1$ and such that there is a compact essential support of $\DD'$ contained in $M_1$. 
\end{lemma}

\begin{proof}
By Proposition~5.4 of \cite{BrMaschler19} there exists a smooth deformation $(c_t,\n^E_t)$ of the Clifford multiplication and the Clifford connection such that 
\begin{enumerate}
\item
for $t=0$ it is equal to $(c,\n^E)$;
\item for $t>0$ it is a product near $N_1$;
\item  for all $t$ its restriction to $M\backslash U$ is independent of $t$ (and, hence, coincides with $(c,\n^E)$).
\end{enumerate}
Moreover, since all our structures are product near $N_1\backslash K$, the construction of this deformation in Appendix~A of \cite{BrMaschler19} provides a deformation which is independent of $t$ on $M\backslash(U\cap K)$. Thus for all $t>0$ the Dirac operator $D_t$ defined by $(c_t,\n^E_t)$ is a  compact perturbation of $D$. 

Let $\Phi_t(x)$ be a smooth deformation of $\Phi(x)$ which coincides with $\Phi$ at $t=0$, is independent of $t$ for $x\not\in U\cap K$, and is product near $N_1$ for all $t>0$. Then $\DD_t:= D_t+i\Phi_t$ is a compact perturbation of $\DD$ for all $t\ge0$. 

Fix $R>0$  such that there is an $R$-essential support of $\DD$ which is contained in $M_1$. Then there exists a compact set $K_R\subset M_1$ such that outside of $K_R$ the estimate \eqref{E:strinvinf} holds. Since all our deformations are smooth and compactly supported $\Phi^2_t(x)\;-\;|[D_t,\Phi_t](x)|\;\ge\;R/2$ for all small enough $t>0$. The assertion of the lemma holds now with $R'=R/2$ and $\DD'= \DD_t$ with $t>0$ small. 
\end{proof}

\subsection{Reduction of the index problem to an almost compact essential support}\label{SS:M=M1}
Let $M_1\subset M$ be an almost compact essential support of $\DD$. Let $\DD'$ be as in the previous subsection. Let $\AAA$ be the restriction of $\DD$ to $\pM$. It is also the restriction of  $\DD'$ (since $\DD'= \DD$ near $\pM$). We denote by $-\AAA_1$ the restriction of $\DD'$ to $N_1$. Thus near $N_1$ the operator $\DD'$ has the form $c(\tau)(\p_t-\AAA_1)$. The sign convention is related to the fact that it is often convenient to view $N_0=\dm$ as the ``left" boundary of $M_1$ and $N_1$  as the ``right" boundary. Then one identifies a neighborhood of $N_1$ in $M_1$ with the product $(-r,0]\times{}N_1$. With respect to this identification the restriction of $\DD'$ to this neighborhood becomes $c(dt)(\p_t+\AAA_1)$. In particular, on  the cylindrical part $M_1\backslash{K}$ we have $\AAA_1= \AAA$.

\begin{theorem}\label{T:indM=indM1}
Suppose $M_1\subset M$ is an almost compact essential support of\/ $\DD$ and let\/ $\pM_1=\pM\sqcup N_1$.  Let\/ $\DD'$ be a compact perturbation of\/ $\DD$ which is product near $N_1$ and  such that there is a compact essential support for $\DD'$ which is contained in $M_1$. Let $B_0$ be any elliptic boundary condition for $\DD$  and let 
\[
	B_1 \ = \ H_{(-\infty,0)}^{1/2}(-\AAA_1) \ = \ 
	H_{(0,\infty)}^{1/2}(\AAA_1)
\] 
be the APS boundary condition for the restriction of\/ $\DD'$ to a neighborhood of $N_1$. Then $B_0\oplus B_1$ is an elliptic boundary condition for the restriction $\DD'':= \DD'|_{M_1}$ of\/ $\DD'$ to $M_1$ and  
\begin{equation}\label{E:indM=indM1}
	\ind \DD_{B_0}\ = \ \ind \DD''_{B_0\oplus B_1}.
\end{equation}
\end{theorem}
\begin{proof}
Let $\DD'''$ denote the restriction of $\DD'$ to $M\backslash{M_1}$. This is a strongly Callias-type operator with an empty essential support. Notice that its restriction to $N_1$ is equal to $\AAA_1$. Thus the APS boundary condition for $\DD'''$ is $B_2 =  H_{(-\infty,0)}^{1/2}(\AAA_1)$. Since $\AAA_1$ is invertible, $B_2$ coincides with the dual APS boundary condition for $\DD'$. Hence, by the Splitting Theorem~\ref{T:splitting},
\[
	\ind \DD'_{B_0} \ = \ 
	\ind \DD''_{B_0\oplus B_1} \ + \ \ind \DD'''_{B_2}.
\]
The last summand on the right hand side of this equality vanishes by Corollary~\ref{C:vanishing}. The theorem follows now from \eqref{E:D=D'}. 
\end{proof}

\section{The index of operators on essentially cylindrical manifolds}\label{S:esscyl}

In the previous section we reduced the computation of the index of $\DD$ to a computation of the index of the restriction of $\DD$ to its almost compact essential support (which is an essentially cylindrical manifold). In this section we consider a strongly  Callias-type operator $\DD$ on an essentially cylindrical manifold $M$  (these data might or might not come as a restriction of another operator to its almost compact essential support. In particular, we don't assume that the restriction of $\DD$ to $N_1$ is invertible).  From this point on we assume that {\em the dimension of $M$ is odd}.

Let $\AAA_0$ and $-\AAA_1$ be the restrictions of $\DD$ to $N_0$ and $N_1$ respectively. The main result of this section is that the index of the APS boundary value problem for $\DD$  depends only  on $\AAA_0$ and $\AAA_1$. Thus it is an invariant of the boundary.  In the next section we will discuss the properties of this invariant.

\subsection{Compatible essentially cylindrical manifolds}\label{SS:essequal}

Let $M$ be an essentially cylindrical manifold and let $\pM=N_0\sqcup N_1$.  As usual, we identify a tubular neighborhood of $\pM$ with the product 
\[
	Z_r\ := \  \big(\, N_0\times[0,r)\,\big)\,\sqcup \, \big(\, N_1\times[0,r)\,\big)
	\ \subset \ M.
\]

\begin{definition}\label{D:essequal}
We say that another essentially cylindrical manifold $M'$ is {\em compatible} with  $M$ if there is a fixed isometry between $Z_r$ and a neighborhood $Z_r'\subset M'$ of the boundary of $M'$. 
\end{definition}

Note that if $M$ and $M'$ are compatible then their boundaries are isometric. 

\subsection{Compatible strongly Callias-type operators}\label{SS:essequal operators}
Let $M$ and $M'$ be compatible essentially cylindrical manifolds and let $Z_r$ and $Z_r'$ be as above. Let $E\to M$ be a Dirac bundle over $M$ and let  $\DD: C^\infty(M,E)\to C^\infty(M,E)$  be a strongly Callias-type operator whose restriction to $Z_r$ is product and such that $M$ is an almost compact essential support of $\DD$.  This  means that there is a compact set $K\subset M$ such that $M\backslash K= [0,\varepsilon]\times N$ and the restriction of $\DD$ to $M\backslash K$ is product (i.e.  is given by \eqref{E:productDD}). Let 
 $E'\to M'$ be a Dirac bundle over $M'$ and let $\DD': C^\infty(M',E')\to C^\infty(M',E')$ be a strongly Callias-type operator, whose restriction to  $Z_r'$ is product and such that  $M'$ is an almost compact essential support of $\DD'$.

\begin{definition}\label{D:essequal operators}
In the situation discussed above we say that $\DD$ and $\DD'$ are {\em compatible} if there is an isomorphism $E|_{Z_r}\simeq E'|_{Z_r'}$  which identifies the  restriction of $\DD$ to $Z_r$ with the restriction of $\DD'$ to $Z_r'$.  
\end{definition}

Let $\AAA_0$ and $-\AAA_1$ be the restrictions of $\DD$ to $N_0$ and $N_1$ respectively. 
Let $B_0=H^{1/2}_{(-\infty,0)}(\AAA_0)$ and $B_1=H^{1/2}_{(-\infty,0)}(-\AAA_1)= H^{1/2}_{(0,\infty)}(\AAA_1)$ be the APS boundary conditions for $\DD$ at $N_0$ and $N_1$ respectively. Since $\DD$ and $\DD'$ are equal near the boundary, $B_0$ and $B_1$ are also elliptic boundary conditions for $\DD'$.

\begin{theorem}\label{T:indep of D}
Suppose $\DD$ is a strongly Callias-type operator on an essentially cylindrical odd-dimensional manifold $M$ such that $M$ is an almost compact essential support of\, $\DD$. Suppose that the operator  $\DD'$ is compatible with $\DD$. Let $\pM=N_0\sqcup N_1$ and let $B_0= H^{1/2}_{(-\infty,0)}(\AAA_0)$ and $B_1= H^{1/2}_{(-\infty,0)}(-\AAA_1)= H^{1/2}_{(0,\infty)}(\AAA_1)$ be the APS boundary conditions for $\DD$ (and, hence, for $\DD'$) at  $N_0$ and $N_1$ respectively. Then 
\begin{equation}\label{E:indep of D}
	\ind \DD_{B_0\oplus B_1} \ = \ \ind \DD'_{B_0\oplus B_1}.
\end{equation}
\end{theorem}

The proof of this theorem occupies Subsections~\ref{SS:tilM}--\ref{SS:prindep of D}.

\subsection{Gluing together $M$ and $M'$}\label{SS:tilM}

Let $-M'$ denote  another copy of  manifold $M'$. Even though we don't assume that our manifolds are oriented, it is useful to think of $-M'$ as manifold $M$ with the opposite orientation. We identify the boundary of $-M'$ with the product 
\[
	-Z_r'\ := \ 
	\big(\, N_0\times(-r,0]\,\big)\,\sqcup \, \big(\, N_1\times(-r,0]\,\big)
\]
and consider the union 
\[
	\tilde{M} \ := \ M\cup_{N_0\sqcup N_1} (-M').
\]
Then $Z_{(-r,r)}:= Z_r\cup (-Z_r')$ is a subset of $\tilde{M}$ identified with the product 
\[
	\big(\, N_0\times(-r,r)\,\big)\,\sqcup \, \big(\, N_1\times(-r,r)\,\big).
\]
We note that $\tilde{M}$ is a complete Riemannian manifold without boundary.

\subsection{Gluing together $\DD$ and $(\DD')^*$}\label{SS:tilDD}

Let $E_\pM$ denote the restriction of $E$ to $\pM$. The product structure on $E|_{Z_r}$ gives an isomorphism $\psi:E|_{Z_r}\to [0,r)\times E_\pM$. Recall that we identified $Z_r$ with $Z_r'$ and fixed an isomorphism between the restrictions of $E$ to $Z_r$ and $E'$ to $Z_r'$. By a slight abuse of notation we use this isomorphism to view $\psi$ also as an isomorphism $E'|_{Z'_r}\to [0,r)\times E_\pM$.

Let $\tilde{E}\to \tilde{M}$ be the vector bundle over $\tilde{M}$ obtained by gluing $E$ and $E'$ using the isomorphism $c(\tau):E|_\pM\to E'|_{\pM'}$. This means that we fix isomorphisms 
\begin{equation}\label{E:tilE=}
	\phi:\,\tilde{E}|_M\to E, \quad \phi':\,\tilde{E}|_{M'}\to E', 
\end{equation}
so that 
\[
	\begin{aligned}
	\psi\circ\phi\circ \psi^{-1} \ &= \ 
		\id: [0,r)\times E_\pM \ \to \ [0,r)\times E_\pM, \\
	\psi\circ\phi'\circ \psi^{-1} \ &= \ 
		1\times c(\tau): [0,r)\times E_\pM \ \to \ [0,r)\times E_\pM.
	\end{aligned}
\]

We denote by $c':T^*M'\to \End(E')$ the Clifford multiplication on $E'$ and set $c''(\xi):= -c'(\xi)$. We think of $c''$ as the Clifford multiplication of $T^*(-M')$ on $E'$ (since the dimension of $M'$ is odd, changing the sign of the Clifford multiplication corresponds to changing the orientation on $M'$). Then $\tilde{E}$ is a Dirac bundle over $\tilde{M}$ with the Clifford multiplication
\begin{equation}\label{E:tildec}
	\tilde{c}(\xi)\ : = \ 
  \begin{cases}
  	&c(\xi), \qquad \xi\in T^*M;\\
  	&c''(\xi)=-c'(\xi), \qquad \xi\in T^*M'.
  \end{cases}
\end{equation}
One readily checks that \eqref{E:tildec} defines a smooth Clifford multiplication on $\tilde{E}$. Let $\tilde{D}:C^\infty(\tilM,\tilE)\to C^\infty(\tilM,\tilE)$ be the Dirac operator. Then the isomorphism $\phi$ of \eqref{E:tilE=} identifies the restriction of $\tilde{D}$ with $D$, the isomorphism $\phi'$ identifies the restriction of $\tilde{D}$ with $-D'$, and isomorphism $\psi\circ\phi'\circ\psi^{-1}$ identifies the restriction of $\tilde{D}$ to $-Z_r'$ with
\begin{equation}\label{E:tilDM'}
 	\tilde{D}|_{Z_r'}\ = \ - c'(\tau)\circ D'_{Z'_r}\circ c'(\tau)^{-1}.
 \end{equation} 

Let $\Phi'$ denote the Callias potential of $\DD'$, so that $\DD'=D'+i\Phi'$. Consider the bundle map $\tilde{\Phi}\in \End(\tilde{E})$ whose restriction to $M$ is equal to $\Phi$ and whose restriction to $M'$ is equal to $\Phi'$. We summarize the constructions presented in this subsection in the following

\begin{lemma}\label{L:tildeD}
The operator $\tilde{\DD} :=  \tilde{D} + i\tilde{\Phi}$ 
is a strongly Callias-type operator on $\tilde{M}$, whose restriction to $M$ is equal to $\DD$ and whose restriction to $M'$ is equal to $-D'+i\Phi'= -(\DD')^*$.
\end{lemma}

The operator $\tilde{\DD}$ is a strongly Callias-type operator on a complete Riemannian manifold without boundary. Hence, \cite{Anghel90}, it is Fredholm. 

\begin{lemma}\label{L:indtildeD}
$\ind \tilde{\DD}=0$. 
\end{lemma}

\begin{proof}
Since $\tilde{M}$ is a union of two essentially cylindrical manifolds, there exists a compact essential support $\tilde{K}\subset \tilde{M}$ of $\tilde{\DD}$ such that $\tilde{M}\backslash \tilde{K}$ is of the form $S^1\times N$. We can choose  $\tilde{K}$ to be large enough so that the restriction of $\tilde{\DD}$ to $S^1\times N$ is a product. We can also assume that $\tilde{K}$ has a smooth boundary $\Sigma=S^1\times L$. Then the Callias index theorem  \cite[Theorem~1.5]{Anghel93Callias} states that the index of $\tilde{\DD}$ is equal to the index of a certain operator $\p$ on $\Sigma$. Since all our structures are product on $\tilde{M}\backslash \tilde{K}$, the operator $\p$ is also a product on $\Sigma=S^1\times L$. Thus it has a form 
\[
	\p \ = \ \gamma\,\big(\p_t+ \tilde{A}\big),
\]
where $\tilde{A}$ is an operator on $L$. The kernel and cokernel of $\p$ can be computed by  separation of variables and both are easily seen to be isomorphic to the kernel of $\tilde{A}$. Thus the kernel and the cokernel are isomorphic and  $\ind\p=0$. 
\end{proof}

\subsection{Proof of Theorem~\ref{T:indep of D}}\label{SS:prindep of D}

Recall that we denote by $B_0$ and $B_1$ the APS boundary conditions for $\DD= \tilde{\DD}|_{M}$. Let $\DD''$ denote the restriction of $\tilde\DD$ to $-M'=\tilde{M}\backslash M$ and let $B_0'$ and $B_1'$ be the dual APS boundary conditions for $\DD''$ at $N_0$ and $N_1$ respectively. By the Splitting Theorem~\ref{T:splitting} 
\begin{equation}\label{E:tilD=D+D'}\notag
	\ind\tilde{\DD} \ = \ 
	\ind \DD_{B_0\oplus B_1}\ + \ \ind \DD''_{B'_0\oplus B'_1}.
\end{equation}
Since, by Lemma~\ref{L:indtildeD}, $\ind\tilde{\DD}=0$, we obtain 
\begin{equation}\label{E:D=-D''}
	\ind \DD_{B_0\oplus B_1}\ = \ -\,\ind \DD''_{B'_0\oplus B'_1}.
\end{equation}

By Lemma~\ref{L:tildeD}, $\DD''= -(\DD')^*$. Thus, by Remark~\ref{R:aps-dual aps}, $B_0'\oplus B_1'$ is equal to the adjoint boundary conditions for $-\DD'$. Hence, by \eqref{E:indD-indD*},
\[
	\ind \DD''_{B'_0\oplus B'_1}\ = \ \ind (-\DD')^*_{B'_0\oplus B'_1}\ = \ 
	-\, \ind \DD'_{B_0\oplus B_1}.
\]
Combining this equality with \eqref{E:D=-D''} we obtain \eqref{E:indep of D}. \hfill$\square$

\section{The relative $\eta$-invariant}\label{S:releta}

In the previous section we proved that the index of the APS boundary value problem $\DD_{B_0\oplus B_1}$ for a strongly Callias-type operator on an odd-dimensional  essentially cylindrical manifold depends only on the restriction of $\DD$ to the boundary, i.e. on the  operators $\AAA_0$ and $-\AAA_1$. In this section we use this index to define the {\em relative $\eta$-invariant} $\eta(\AAA_1,\AAA_0)$ and show that it has properties similar to the difference of $\eta$-invariants $\eta(\AAA_1)-\eta(\AAA_0)$ of operators on compact manifolds. For special cases, \cite{FoxHaskell05}, when the index can be computed using heat kernel asymptotics, we show that  $\eta(\AAA_1,\AAA_0)$ is indeed equal to the difference of the $\eta$-invariants of $\AAA_1$ and $\AAA_0$. In the next section we discuss the connection between the relative $\eta$-invariant and the spectral flow.

\subsection{Almost compact cobordisms}\label{SS:almost compact cobordism}
Let $N_0$ and $N_1$ be two complete {\em even-dimensional} Riemannian manifolds and let $\AAA_0$ and $\AAA_1$ be self-adjoint strongly Callias-type operators on $N_0$ and $N_1$ respectively, cf. Definition~\ref{D:sa_stronglyCallias}.

\begin{definition}\label{D:almost compact cobordism}
An {\em almost compact cobordism} between $\AAA_0$ and $\AAA_1$ is given by an essentially cylindrical manifold $M$ with $\pM=N_0\sqcup N_1$ and a strongly Callias-type operator $\DD$ on $M$ such that 
\begin{enumerate}
\item $M$ is an almost compact essential support of $\DD$;
\item $\DD$ is product near $\pM$;
\item 	The restriction of $\DD$ to $N_0$ is equal to $\AAA_0$ and the restriction of $\DD$ to $N_1$ is equal to $-\AAA_1$.
\end{enumerate}
If there exists an almost compact cobordism between $\AAA_0$ and $\AAA_1$ we say that operator $\AAA_0$ is {\em cobordant} to operator $\AAA_1$.
\end{definition}

\begin{lemma}\label{L:antisymmetry}
If\/ $\AAA_0$ is cobordant to $\AAA_1$ then $\AAA_1$ is cobordant to $\AAA_0$. 
\end{lemma}

\begin{proof}
Let $-M$ denote the manifold $M$ with the opposite orientation and let $\tilde{M}:= M\cup_\pM(-M)$ denote the {\em double} of $M$. Let $\DD$ be an almost compact cobordism between $\AAA_0$ and $\AAA_1$. Using the construction of  Section~\ref{SS:tilDD} (with $\DD'= \DD$) we obtain a strongly Callias-type operator $\tilde{\DD}$ on $\tilde{M}$ whose restriction to $M$ is isometric to $\DD$. Let $\DD''$ denote the restriction of $\tilde{\DD}$ to $-M=\tilde{M}\backslash{}M$. Then the restriction of $\DD''$ to $N_1$ is equal to $\AAA_1$ and the restriction of $\DD''$ to $N_0$ is equal to $-\AAA_0$. Hence, $\DD''$ is an almost compact cobordism between $\AAA_1$ and $\AAA_0$. 
\end{proof}

\begin{lemma}\label{L:transitivity}
Let $\AAA_0,\AAA_1$ and $\AAA_2$ be self-adjoint strongly Callias-type operators on even-dimensional complete Riemannian manifolds $N_0, N_1$ and $N_2$ respectively. 
Suppose $\AAA_0$ is cobordant to $\AAA_1$ and $\AAA_1$ is cobordant to $\AAA_2$. Then $\AAA_0$ is cobordant to $\AAA_2$.
\end{lemma}
\begin{proof}
Let $M_1$ and $M_2$ be essentially  cylindrical manifolds such that $\pM_1= N_0\sqcup N_1$ and $\pM_2= N_1\sqcup N_2$. Let $\DD_1$ be an operator on $M_1$ which is an almost compact cobordism between $\AAA_0$ and $\AAA_1$. Let $\DD_2$ be an operator on $M_2$ which is an almost compact cobordism between $\AAA_1$ and $\AAA_2$. Then the operator $\DD_3$ on $M_1\cup_{N_1} M_2$ whose restriction to $M_j$ ($j=1,2$) is equal to $\DD_j$ is an almost compact cobordism between $\AAA_0$ and $\AAA_2$. 
\end{proof}

If follows  from Lemmas~\ref{L:antisymmetry} and  \ref{L:transitivity} that {\em cobordism  is an equivalence relation on the set of self-adjoint strongly Callias-type operators}. 
\begin{definition}\label{D:releta}
Suppose $\AAA_0$ and $\AAA_1$ are cobordant self-adjoint strongly Callias-type operators and let $\DD$ be an almost compact cobordism between them. Let $B_0= H_{(-\infty,0)}^{1/2}(\AAA_0)$ and $B_1= H_{(-\infty,0)}^{1/2}(-\AAA_1)$ be the APS boundary conditions for $\DD$. The {\em relative $\eta$-invariant} is defined as 
\begin{equation}\label{E:releta}
	\eta(\AAA_1,\AAA_0) \ = \  2\,\ind \DD_{B_0\oplus B_1}
	 \ + \ \dim\ker \AAA_0\ + \ \dim\ker \AAA_1.
\end{equation}
\end{definition}
Theorem~\ref{T:indep of D}  implies that  $\eta(\AAA_1,\AAA_0)$ is independent of the choice of the cobordism $\DD$.

\begin{remark}\label{R:rel eta}
Sometimes it is convenient to use the dual APS boundary conditions $\oB_0= H_{(-\infty,0]}^{1/2}(\AAA_0)$ and $\oB_2= H_{(-\infty,0]}^{1/2}(-\AAA_1)$ instead of $B_0$ and $B_1$. It follows from Corollary~\ref{C:change of B} that the relative $\eta$-invariant can be written as
\begin{equation}\label{E:B1'B2'}
		\eta(\AAA_1,\AAA_0) \ = \  2\,\ind \DD_{\oB_0\oplus \oB_1}
	 \ - \ \dim\ker \AAA_0\ - \ \dim\ker \AAA_1.
\end{equation}
\end{remark}

\subsection{The case when the heat kernel has an asymtotic expansion}\label{SS:FoxHaskell}
In \cite{FoxHaskell05}, Fox and Haskell studied the index of a boundary value problem on manifolds of bounded geometry. They showed that under certain conditions (satisfied for natural operators on manifolds with conical or cylindrical ends) on $M$ and $\DD$,  the heat kernel $e^{-t(\DD_B)^*\DD_B}$ is of trace class and its trace has an asymptotic expansion similar to the one on compact manifolds. In this case the $\eta$-function, defined by a usual formula 
\[
		\eta(s;\AAA) \ := \ \sum_{\lambda\in{\rm spec}(\AAA)} 
		{\rm sign}(\lambda)\,|\lambda|^s, 
		\qquad \operatorname{Re} s \ll 0,
\] 
is an analytic function of $s$, which has a meromorphic continuation to the whole complex plane and is regular at 0. So one can define the $\eta$-invariant of $\AAA$ by $\eta(\AAA)= \eta(0;\AAA)$.

\begin{proposition}\label{P:FoxHaskell}
Suppose now that $\DD$ is an operator on an essentially cylindrical manifold $M$ which satisfies the conditions of\, \cite{FoxHaskell05}. We also assume that $\DD$ is product near $\pM=N_0\sqcup N_1$ and that $M$ is an almost compact essential support for $\DD$. Let $\AAA_0$ and $-\AAA_1$ be the restrictions of $\DD$ to $N_0$ and $N_1$ respectively. Let $\eta(\AAA_j)$ ($j=0,1$) be the $\eta$-invariant of $\AAA_j$.   Then 
\begin{equation}\label{E:FoxHaskell}
	\eta(\AAA_1,\AAA_0) \ = \ \eta(\AAA_1)\ - \ \eta(\AAA_0).
\end{equation}
\end{proposition}

\begin{proof}
Theorem~9.6 of \cite{FoxHaskell05} establishes an index theorem for the APS boundary value problem satisfying conditions discussed above. This theorem is completely analogous to the classical APS index theorem \cite{APS1}. In \cite{FoxHaskell05} only the case of even-dimensional manifolds is discussed. However, exactly the same (but somewhat simpler) arguments give an index theorem on  odd-dimensional manifolds as well. In the odd-dimensional case the integral term in the index formula vanishes identically. Thus, applied to our situation, Theorem~9.6 of \cite{FoxHaskell05} gives 
\[
	\ind \DD_{B_0\oplus B_1}\ =\ 
	-\,\frac{\dim\ker \AAA_0+\eta(\AAA_0)}{2} \ - \
	 \frac{\dim\ker \AAA_1+\eta(-\AAA_1)}{2}.
\]
Since $\eta(-\AAA_1)= -\eta(\AAA_1)$ equation \eqref{E:FoxHaskell} follows now from the definition \eqref{E:releta} of the relative $\eta$-invariant. 
\end{proof}

More generally, Bunke, \cite{Bunke92}, considered the situation when $\AAA_je^{-t\AAA_j^2}$  ($j=0,1$) are not of trace class but their difference $\AAA_1e^{-t\AAA_1^2}- \AAA_0e^{-t\AAA_0^2}$ is of trace class and its trace has a nice asymptotic expansion. In this situation one can define the relative $\eta$-function by the usual formula
\begin{equation}\label{E:reletafunction}
		\eta(s;\AAA_1,\AAA_0) \ := \
	\frac{1}{\Gamma\big((s+1)/2\big)}\,\int_0^\infty\,
	  t^{\frac{s-1}{2}}\,
	     \Tr\,\big(\, \AAA_1e^{-t\AAA_1^2}- \AAA_0e^{-t\AAA_0^2}\,\big) \,dt.
\end{equation}
(See \cite{Muller98} for even more general situation when the relative $\eta$-function can be defined.) 

Bunke only considered the undeformed Dirac operator $A$ and gave a geometric condition under which $\Tr(A_1e^{-tA_1^2}- A_0e^{-tA_0^2})$ has a nice asymptotic expansion and the  above integral gives a meromorphic function regular at 0. One can also consider the cases when the heat kernels of the Callias-type operators $\AAA_j$ are such that $\Tr(\AAA_1e^{-t\AAA_1^2}- \AAA_0e^{-t\AAA_0^2})$ has a nice asymptotic expansion and the relative $\eta$-function can be defined using \eqref{E:reletafunction}.

\begin{conjecture}\label{Con:releta}
If the relative $\eta$-function \eqref{E:reletafunction} is defined, analytic and regular at 0, then 
\begin{equation}\label{E:reletaconjecture}
	\eta(\AAA_1,\AAA_0) \ = \ \eta(0;\AAA_1,\AAA_0).
\end{equation}
\end{conjecture}

\subsection{Basic properties of the relative $\eta$-invariant}\label{SS:properties eta}

Proposition~\ref{P:FoxHaskell} shows that under certain conditions the $\eta$-invariants of $\AAA_0$ and $\AAA_1$ are defined and $\eta(\AAA_1,\AAA_0)$ is their difference. We now show that in general case, when $\eta(\AAA_0)$ and $\eta(\AAA_1)$ do not necessarily exist,  $\eta(\AAA_1,\AAA_0)$ behaves like it was a difference of an invariant of $N_1$ and an invariant of $N_0$.

\begin{proposition}[Antisymmetry]\label{P:antisymmetry eta}	
Suppose $\AAA_0$ and $\AAA_1$ are cobordant self-adjoint strongly Callias-type operators. Then 
\begin{equation}\label{E:antisymmetry}
	\eta(\AAA_0,\AAA_1)\ = \ -\,\eta(\AAA_1,\AAA_0).
\end{equation}
\end{proposition}

\begin{proof}
Let $\DD$ be an almost compact cobordism between $\AAA_0$ and $\AAA_1$ and let $\tilde{\DD}$ and $\DD''$ be as in the proof of Lemma~\ref{L:antisymmetry}. Then $\tilde{\DD}$ is a strongly Callias-type operator on a complete Riemannian manifold $\tilde{M}$ without boundary and $\DD''$ is an almost compact cobordism between $\AAA_1$ and $\AAA_0$. 

Let   
\[
	\begin{aligned}
	B_0' \ &= \ H^{1/2}_{[0,\infty)}(\AAA_0)\ = \  H^{1/2}_{(-\infty,0]}(-\AAA_0);\\
	B_1' \ &= \  H^{1/2}_{[0,\infty)}(-\AAA_1) \ = \ H^{1/2}_{(-\infty,0]}(\AAA_1).
	\end{aligned}
\] 
be the dual APS boundary conditions for $\DD''$.  It is shown in Section~\ref{SS:prindep of D} that $B_0'\oplus B_1'$ is an elliptic boundary condition for $\DD''$ and, by \eqref{E:D=-D''}, 
\begin{equation}\label{E:indD''=indD}
		\ind \DD''_{B_0^{\prime}\oplus B_1^{\prime}} \ = \ 
	-\ind \DD_{B_0\oplus B_1}.
\end{equation}

Since $\DD''$ is an almost compact cobordism between $\AAA_1$ and $\AAA_0$ we conclude  from \eqref{E:B1'B2'} that
\begin{equation}\label{E:etaA0A1}
	\eta(\AAA_0,\AAA_1)\ = \ 
	2\,\ind \DD''_{B_0^\prime\oplus B_1^\prime} 
	\ - \ \dim\ker \AAA_0\ - \ \dim\ker \AAA_1.
\end{equation}
Combining  \eqref{E:etaA0A1} and \eqref{E:indD''=indD} we obtain \eqref{E:antisymmetry}.
\end{proof}

Note that \eqref{E:antisymmetry} implies that 
\begin{equation}\label{E:etaAA}
	\eta(\AAA,\AAA)\ = \ 0
\end{equation}
for every self-adjoint strongly Callias-type operator $\AAA$. 

\begin{proposition}[The cocycle condition]\label{P:cocycle}
Let $\AAA_0,\AAA_1$ and $\AAA_2$ be self-adjoint strongly Callias-type operators which are cobordant to each other. Then 
\begin{equation}\label{E:cocycle}
	\eta(\AAA_2,\AAA_0)\ = \ \eta(\AAA_2,\AAA_1)\ + \ \eta(\AAA_1,\AAA_0).
\end{equation}
\end{proposition}
\begin{proof}
The lemma follows from the Splitting Theorem~\ref{T:splitting} applied to the operator $\DD_3$ constructed in the proof of Lemma~\ref{L:transitivity}. 
\end{proof}

\section{The spectral flow}\label{S:sp flow}

Atiyah, Patodi and Singer, \cite{APS3}, introduced a notion of spectral flow $\spf(\AA)$ of a continuous family $\AA:= \{\AAA^s\}_{0\le s\le 1}$ of self-adjoint differential operators on a closed manifold. They showed that the spectral flow  computes the variation of the $\eta$-invariant $\eta(\AAA^1)-\eta(\AAA^0)$. In this section we consider a family of self-adjoint strongly Callias-type operators $\AA= \{\AAA^s\}_{0\le s\le 1}$ on a complete {\em even-dimensional} Riemannian manifold and show that for any operator $\AAA_0$ cobordant to $\AAA^0$ we have $\eta(\AAA^1,\AAA_0)- \eta(\AAA^0,\AAA_0)= 2\,\spf(\AA)$.

\subsection{A family of boundary operators}\label{SS:familyAA}
Let $E_N\to N$ be a Dirac bundle over a complete {\em even-dimensional} Riemannian manifold $N$. Let $\AA= \{\AAA^s\}_{0\le s\le 1}$ be a family of self-adjoint strongly Callias-type operators 
\[
	\AAA^s=A^s+i\,\Psi^s:\,C^\infty(N,E_N)\ \to\ C^\infty(N,E_N).
\]

\begin{definition}\label{D:almost constant}
The family $\AA= \{\AAA^s\}_{0\le s\le 1}$ is called {\em almost constant} if there exists a compact set $K\subset N$ such that the restriction of $\AAA^s$ to $N\backslash K$ is independent of $s$. 
\end{definition}

Since $\dim N= 2p$ is even, there is a natural grading operator $\Gamma:E_N\to E_N$, with $\Gamma^2=1$, cf. \cite[Lemma~3.17]{BeGeVe}. If $e_1,\ldots, e_{2p}$ is an orthonormal basis of $TN\simeq T^*N$, then 
\[
	\Gamma\ := \ i^p\,c(e_1)\cdots c(e_{2p}).
\]

\begin{remark}\label{R:spectral asymetry}
The operators $A^s$ anticommute with $\Gamma$. Condition (i) of  Definition~\ref{D:sa_stronglyCallias} implies that $\Psi$ anticommutes with $c(e_j)$ ($j=1,\ldots,2p$) and, hence, commutes with $\Gamma$. So the operators $\AAA^s$  neither commute, nor anticommute with $\Gamma$. This explains why, even though the dimension of $N$ is even, the spectrum of the operators $\AAA^s$ is not symmetric about the origin and the spectral flow of the family $\AA$ is, in general, not trivial. 
\end{remark}

We set $M:=[0,1]\times N$, $E:= [0,1]\times E_N$ and denote by $t$ the coordinate along $[0,1]$. Then $E\to M$ is naturally a Dirac bundle over $M$ with $c(dt):= i\Gamma$.

\begin{definition}\label{D:smoothD}
The family $\AA= \{\AAA^s\}_{0\le s\le 1}$ is called {\em smooth} if  
\[
	\DD\ : = \ c(dt)\,\big(\p_t+\AAA^t\,\big)\ :\ 
	C^\infty(M,E)\to C^\infty(M,E)
\]
is a smooth differential operator on $M$. 
\end{definition}

Fix a smooth non-decreasing function $\kappa:[0,1]\to [0,1]$ such that $\kappa(t)=0$ for $t\le1/3$ and $\kappa(t)=1$ for $t\ge2/3$ and consider the operator 
\begin{equation}\label{E:interpolationD}
	\DD\ : = \ c(dt)\,\big(\p_t+\AAA^{\kappa(t)}\,\big)\ :\ 
	C^\infty(M,E)\to C^\infty(M,E).
\end{equation}
Then $\DD$ is product near $\pM$. If $\AA$ is a smooth almost constant family of self-adjoint strongly Callias-type operators then \eqref{E:interpolationD} is a strongly Callias-type operator for which $M$ is an almost compact essential support. Hence it is a non-compact cobordism (cf. Definition~\ref{D:almost compact cobordism}) between $\AAA^0$ and $\AAA^1$. 

\subsection{The spectral section}\label{SS:sp section}

If\/ $\AA= \{\AAA^s\}_{0\le s\le 1}$ is a smooth almost constant family of self-adjoint strongly Callias-type operators then it satisfies the conditions of the Kato Selection Theorem \cite[Theorems~II.5.4 and II.6.8]{Kato95book}, \cite[Theorem~3.2]{Nicolaescu95}. Thus there is a family of eigenvalues $\lambda_j(s)$ ($j\in \ZZ$)  which depend continuously on $s$. We order the eigenvalues so that $\lambda_j(0)\le \lambda_{j+1}(0)$ for all $j\in \ZZ$ and $\lambda_j(0)\le 0$ for $j\le 0$ while $\lambda_j(0)>0$ for $j>0$. 

Atiyah, Patodi and Singer \cite{APS3} defined the spectral flow $\spf(\AA)$ for a family of operators satisfying the conditions of the Kato Selection Theorem  as an integer that counts the net number of eigenvalues that change sign when $s$ changes from 0 to 1. 
Several other equivalent definitions of the spectral flow based on different assumptions on the family $\AA$ exist in the literature. For our purposes the most convenient is the  Dai and Zhang's definition \cite{DaiZhang98} which is based on the notion of {\em spectral section} introduced by Melrose and Piazza \cite{MelrosePiazza97}.

\begin{definition}\label{D:spectral section}
A {\em spectral section} for $\AAA$ is a continuous family $\PP = \{P^s\}_{0\le s\le 1}$ of self-adjoint projections such that there exists a constant $R>0$ such that for all $0\le s\le 1$, if $\AAA^su= \lambda u$ then
\[
	P^su \ = \ \begin{cases} 
		0, \quad &\text{if}\quad \lambda< -R;\\ 
		u, \quad	 &\text{if}\quad \lambda> R.	
		\end{cases}
\]
\end{definition}

If $\AA$ satisfies the conditions of the Kato Selection Theorem, then the arguments of the proof of  \cite[Proposition~1]{MelrosePiazza97} show that $\AA$ admits a spectral section.

\begin{remark}\label{R:B-BLP}
Booss-Bavnbek, Lesch, and Phillips \cite{BoossLeschPhillips05} defined the spectral flow for a family of unbounded operators in an abstract Hilbert space. Their conditions on the family are much weaker than those of the Kato Selection Theorem. In particular, they showed that a family of elliptic differential operators on a closed manifold satisfies their conditions if all the coefficients of the differential operators depend continuously on $s$. It would be interesting to find a good practical condition under which a family of self-adjoint strongly Callias-type operators satisfies the conditions of \cite{BoossLeschPhillips05}.
\end{remark}

\subsection{The spectral flow}\label{SS:sp flow}
Let $\PP = \{P^s\}$ be a spectral section for $\AA$. Set $B^s:= \ker P^s$. Let $B_0^s:= H^{1/2}_{(-\infty,0)}(\AAA^s)$ denote the APS boundary condition defined by the boundary operator $\AAA^s$.  Recall that the relative index of subspaces was defined in Section~\ref{SS:dep on bndry}. Since the spectrum of $\AAA^s$ is discrete, it follows immediately from the definition of the spectral section that for every $s\in[0,1]$ the space $B^s$ is a finite rank perturbation of $B_0^s$. We are interested in the relative index $[B^s,B_0^s]$. Following Dai and Zhang \cite{DaiZhang98} we give the following definition.

\begin{definition}\label{D:sp flow}
Let $\AA= \{\AAA^s\}_{0\le s\le 1}$ be a smooth almost constant  family of self-adjoint strongly Callias-type operators which admits a spectral section $\PP = \{P^s\}_{0\le s\le 1}$. Assume that the operators $\AAA^0$ and $\AAA^1$ are invertible. Let $B^s:= \ker P^s$ and $B_0^s:= H^{1/2}_{(-\infty,0)}(\AAA^s)$. The {\em spectral flow}\/ $\spf(\AA)$ of the family $\AA$ is defined by the formula
\begin{equation}\label{E:sp flow=sp section}
	\spf(\AA)\ := \ [B^1,B_0^1] \ - \ [B^0,B^0_0].
\end{equation}
\end{definition}
By Theorem~1.4 of \cite{DaiZhang98} the spectral flow is independent of the choice of the spectral section $\PP$ and computes the net number of eigenvalues that change sign when $s$ changes from 0 to 1.

\begin{remark}\label{R:rel index DaiZhang}
The relative index $[B^s,B_0^s]$ can also be computed in terms of the orthogonal projections $P^s$ and $P_0^s$ with kernels $B^s$ and $B_0^s$ respectively. Then $P^s_0$ defines a Fredholm operator $P^s_0:\IM P^s\to \IM P_0^s$. Dai and Zhang denote the index of this operator by $[P^s_0-P^s]$ and use it in their formula for spectral flow. One easily checks that $[P^s_0-P^s]= [B^s,B_0^s]$.
\end{remark}

\begin{lemma}\label{L:spflow-A}
Let $-\AA$ denote the family $\{-\AAA^s\}_{0\le s\le 1}$. Then 
\begin{equation}\label{E:spflow-A}
	\spf(-\AA)\  = \ -\spf(\AA).
\end{equation}
\end{lemma}
\begin{proof}
The lemma is an immediate consequence of Lemma~\ref{L:relindex sign}.
\end{proof}

\subsection{Deformation of the relative $\eta$-invariant}\label{SS:def eta}

Let $\AA= \{\AAA^s\}_{0\le s\le 1}$ be a smooth almost constant family of self-adjoint strongly Callias-type operators on a complete even-dimensional Riemannian manifold $N_1$. Let $\AAA_0$ be another self-adjoint strongly Callias-type operator, which is cobordant to $\AAA^0$. In Section~\ref{SS:familyAA} we showed that  $\AAA^0$ is cobordant to $\AAA^s$ for all $s\in [0,1]$. Hence, by Lemma~\ref{L:transitivity}, $\AAA_0$ is cobordant to $\AAA^1$. In this situation we say the $\AAA_0$ is {\em cobordant to the family} $\AA$. The following theorem is the main result of this  section.

\begin{theorem}\label{T:sp flow}
Suppose 
\(\AA  =  \big\{
	  \AAA^s:\,C^\infty(N_1,E_1)\to C^\infty(N_1,E_1)\big\}_{0\le s\le 1}  
\)
is a smooth almost constant family of self-adjoint strongly Callias-type operators on a complete Riemannian manifold $N_1$ such that $\AAA^0$ and $\AAA^1$ are invertible. Then
\begin{equation}\label{E:etaA0A1=spflow}
	\eta(\AAA^1,\AAA^0)\ = \ 2\,\spf(\AA).
\end{equation}

If $\AAA_0:C^\infty(N_0,E_0)\to C^\infty(N_0,E_0)$ is an invertible self-adjoint strongly Callias-type operator on a complete even-dimensional Riemannian manifold $N_0$ which is cobordant to the family $\AA$ then
\begin{equation}\label{E:sp flow}
	\eta(\AAA^1,\AAA_0)\ - \ \eta(\AAA^0,\AAA_0)\ = \ 2\,\spf(\AA).
\end{equation}
\end{theorem}

\begin{proof}
First, we prove \eqref{E:sp flow}. Let $M$ be an essentially cylindrical manifold whose boundary is the disjoint union of $N_0$ and $N_1$. Let $\DD:C^\infty(M,E)\to C^\infty(M,E)$ be an almost compact cobordism between $\AAA_0$ and $\AAA^0$. 

Consider the ``extension of $M$ by a cylinder"
\[
	M'\ := \ M\cup_{N_1} \big([0,1]\times N_1\,\big). 
\]
and let $E'\to M'$ be the bundle over $M'$ whose restriction to $M$ is equal to $E$ and whose restriction to the cylinder $[0,1]\times N_1$ is equal to $[0,1]\times E_1$. 

We fix a smooth function $\rho:[0,1]\times [0,1]\to [0,1]$ such that for each $r\in [0,1]$
\begin{itemize}
\item  the function $s\mapsto \rho(r,s)$ is non-decreasing.
\item
$\rho(r,s)=0$ for $s\le1/3$ and $\rho(r,s)= r$ for $s\ge2/3$.
\end{itemize}
Consider the family of strongly Callias-type operators $\DD^r:C^\infty(M',E')\to C^\infty(M',E')$ whose restriction to $M$ is equal to $\DD$ and whose restriction to $[0,1]\times N_1$ is given by 
\[
	\DD^r\ := \ c(dt)\,\big(\p_t+\AAA^{\rho(r,t)}\,\big).
\]
Then $\DD^r$ is an almost compact cobordism between $\AAA_0$ and $\AAA^r$. In particular, the restriction of $\DD^r$ to $N_1$ is equal to $-\AAA^r$. 

Recall that we denote by $-\AA$ the family $\{-\AAA^s\}_{0\le s\le 1}$.
Let $\PP= \{P^s\}$ be a spectral section for $-\AA$. Then for each $r\in [0,1]$ the space  $B^r:= \ker P^r$ is an elliptic boundary condition for $\DD^r$ at $\{1\}\times N_1$. Let $B_0:= H^{1/2}_{(-\infty,0)}(\AAA_0)$ be the APS boundary condition for $\DD^r$ at $N_0$. Then $B_0\oplus B^r$ is an elliptic boundary condition for $\DD^r$. 

Recall that the domain $\dom \DD^r_{B_0\oplus B^r}$ consists of sections $u$ whose restriction to $\pM'=N_0\sqcup N_1$ lies in $B_0\oplus B^r$.

\begin{lemma}\label{L:indDr}
$\ind \DD^r_{B_0\oplus B^r}= \ind \DD^1_{B_0\oplus B^1}$ for all $r\in [0,1]$. 
\end{lemma}
\begin{proof}
For $r_0,r\in [0,1]$, let $\pi_{r_0r}:\,B^{r_0} \to B^r$
denote the orthogonal projection. Then for every $r_0\in[0,1]$ there exists $\varepsilon>0$ such that if $|r-r_0|<\varepsilon$ then $\pi_{r_0r}$ is an isomorphism. As in the proof of Theorem \ref{T:splitting}, it induces an isomorphism
\[
	\Pi_{r_0r}\,:\,\dom\DD^{r_0}_{B_0\oplus B^{r_0}} \ \to \ 
		\dom\DD^{r}_{B_0\oplus B^{r}}.
\]
Hence
\begin{equation}\label{E:DcircPi}
		\ind \big(\DD^{r}_{B_0\oplus B^{r}}\circ \Pi_{r_0r}\big) 
		\ = \ \ind \DD^{r}_{B_0\oplus B^{r}}.
\end{equation}
Since for $|r-r_0|<\varepsilon$
\[
	\DD^{r}_{B_0\oplus B^{r}}\circ \Pi_{r_0r}:\, \dom\DD^{r_0}_{B_0\oplus B^{r_0}}
	\ \to \ L^2(M',E')
\]
is a continuous family of bounded operators, $\ind (\DD^{r}_{B_0\oplus B^{r}}\circ \Pi_{r_0r})$ is independent of $r$. The lemma follows now from \eqref{E:DcircPi}.
\end{proof}

The space $B^r_0:= H^{1/2}_{(-\infty,0)}(-\AAA^r)$ is the APS boundary conditions for $\DD^r$ at $\{1\}\times N_1$. By definition, $\eta(\AAA^1,\AAA_0)= 2\,\ind \DD^{1}_{B_0\oplus B^1_0}$. To finish the proof of Theorem~\ref{T:sp flow} we note that by Proposition~\ref{P:change of B}
\[
	\ind \DD^{r}_{B_0\oplus B^r} \  = \ \ind \DD^{r}_{B_0\oplus B^r_0} 
	  \ + \ [B^r,B^r_0].
\]
Hence, 
\begin{multline}\notag
	\frac{\eta(\AAA^1,\AAA_0)- \eta(\AAA^0,\AAA_0)\,\big)}2\ = \ 
	\ind \DD^1_{B_0\oplus B^1_0}\ - \ \ind \DD^0_{B_0\oplus B^0_0}
	\\ = \ 
	\left(\,\ind \DD^1_{B_0\oplus B^1}- [B^1,B^1_0]\,\right) \ -\ 
	\left(\, \ind \DD^0_{B_0\oplus B^0}- [B^0,B^0_0]\,\right)
	\\ \overset{\text{Lemma~\ref{L:indDr} }}{=} \
	-[B^1,B^1_0]\ + \ [B^0,B^0_0] \ = \ -\spf(-\AA) 
	\ \overset{\text{Lemma~\ref{L:spflow-A}}}= \ \spf(\AA).
\end{multline}
This proves \eqref{E:sp flow}. Now, by Propostion~\ref{P:cocycle}, 
\[
	\eta(\AAA^1,\AAA^0)\ = \ \eta(\AAA^1,\AAA_0)\ - \ \eta(\AAA^0,\AAA_0)
	 \ = \ 2\,\spf(\AA).
\]
\end{proof}



\begin{bibdiv}
\begin{biblist}

\bib{Anghel90}{article}{
      author={Anghel, N.},
       title={{$L^2$}-index formulae for perturbed {D}irac operators},
        date={1990},
        ISSN={0010-3616},
     journal={Comm. Math. Phys.},
      volume={128},
      number={1},
       pages={77\ndash 97},
         url={http://projecteuclid.org/euclid.cmp/1104180304},
}

\bib{Anghel93}{article}{
      author={Anghel, N.},
       title={An abstract index theorem on noncompact {R}iemannian manifolds},
        date={1993},
        ISSN={0362-1588},
     journal={Houston J. Math.},
      volume={19},
      number={2},
       pages={223\ndash 237},
}

\bib{Anghel93Callias}{article}{
      author={Anghel, N.},
       title={On the index of {C}allias-type operators},
        date={1993},
        ISSN={1016-443X},
     journal={Geom. Funct. Anal.},
      volume={3},
      number={5},
       pages={431\ndash 438},
         url={http://dx.doi.org/10.1007/BF01896237},
}

\bib{APS1}{article}{
      author={Atiyah, M.~F.},
      author={Patodi, V.~K.},
      author={Singer, I.~M.},
       title={Spectral asymmetry and {R}iemannian geometry. {I}},
        date={1975},
     journal={Math. Proc. Cambridge Philos. Soc.},
      volume={77},
      number={1},
       pages={43\ndash 69},
}

\bib{APS3}{article}{
      author={Atiyah, M.~F.},
      author={Patodi, V.~K.},
      author={Singer, I.~M.},
       title={Spectral asymmetry and {R}iemannian geometry. {III}},
        date={1976},
     journal={Math. Proc. Cambridge Philos. Soc.},
      volume={79},
      number={1},
       pages={71\ndash 99},
}

\bib{BaerBallmann12}{incollection}{
      author={B\"ar, C.},
      author={Ballmann, W.},
       title={Boundary value problems for elliptic differential operators of
  first order},
        date={2012},
   booktitle={Surveys in differential geometry. {V}ol. {XVII}},
      series={Surv. Differ. Geom.},
      volume={17},
   publisher={Int. Press, Boston, MA},
       pages={1\ndash 78},
         url={http://dx.doi.org/10.4310/SDG.2012.v17.n1.a1},
}


\bib{BaerBallmann13}{incollection}{
	Author = {B\"{a}r, C.},
	Author = {Ballmann, W.},
	Booktitle = {Arbeitstagung {B}onn 2013},
	Date-Added = {2019-11-24 12:46:38 -0500},
	Date-Modified = {2019-11-24 12:47:28 -0500},
	Mrclass = {58J05 (30G35 35J56 35Q41 58J20 58J32)},
	Mrnumber = {3618047},
	Mrreviewer = {Yuri A. Kordyukov},
	Pages = {43--80},
	Publisher = {Birkh\"{a}user/Springer, Cham},
	Series = {Progr. Math.},
	Title = {Guide to elliptic boundary value problems for {D}irac-type operators},
	Volume = {319},
	Year = {2016},
	Bdsk-Url-1 = {http://arxiv.org/abs/1307.3021}}


\bib{BarStrohmaier15}{article}{
      author={{B\"ar}, C.},
      author={{Strohmaier}, A.},
       title={{An index theorem for Lorentzian manifolds with compact spacelike
  Cauchy boundary}},
        date={2015-06},
     journal={arXiv e-prints: 1506.00959},
      eprint={1506.00959},
}

\bib{BarStrohmaier16}{article}{
      author={B\"ar, C.},
      author={Strohmaier, A.},
       title={A rigorous geometric derivation of the chiral anomaly in curved
  backgrounds},
        date={2016},
        ISSN={0010-3616},
     journal={Comm. Math. Phys.},
      volume={347},
      number={3},
       pages={703\ndash 721},
         url={https://doi.org/10.1007/s00220-016-2664-1},
      review={\MR{3551253}},
}

\bib{BeGeVe}{book}{
      author={Berline, N.},
      author={Getzler, E.},
      author={Vergne, M.},
       title={Heat kernels and {Dirac} operators},
   publisher={Springer-Verlag},
        date={1992},
}

\bib{Bisgard12}{article}{
      author={Bisgard, J.},
       title={A compact embedding for sequence spaces},
        date={2012},
        ISSN={0899-6180},
     journal={Missouri J. Math. Sci.},
      volume={24},
      number={2},
       pages={182\ndash 189},
}

\bib{BoossLeschPhillips05}{article}{
      author={Booss-Bavnbek, B.},
      author={Lesch, M.},
      author={Phillips, J.},
       title={Unbounded {F}redholm operators and spectral flow},
        date={2005},
        ISSN={0008-414X},
     journal={Canad. J. Math.},
      volume={57},
      number={2},
       pages={225\ndash 250},
         url={http://dx.doi.org/10.4153/CJM-2005-010-1},
}

\bib{BoosWoj93book}{book}{
      author={Boo{\ss}-Bavnbek, B.},
      author={Wojciechowski, K.~P.},
       title={Elliptic boundary problems for {D}irac operators},
      series={Mathematics: Theory \& Applications},
   publisher={Birkh\"auser Boston, Inc., Boston, MA},
        date={1993},
        ISBN={0-8176-3681-1},
         url={http://dx.doi.org/10.1007/978-1-4612-0337-7},
}

\bib{BottSeeley78}{article}{
      author={Bott, R.},
      author={Seeley, R.},
       title={Some remarks on the paper of {C}allias: ``{A}xial anomalies and
  index theorems on open spaces'' },
        date={1978},
        ISSN={0010-3616},
     journal={Comm. Math. Phys.},
      volume={62},
      number={3},
       pages={235\ndash 245},
         url={http://projecteuclid.org/euclid.cmp/1103904396},
}

\bib{Br-index}{article}{
      author={Braverman, M.},
       title={Index theorem for equivariant {D}irac operators on noncompact
  manifolds},
        date={2002},
     journal={$K$-Theory},
      volume={27},
      number={1},
       pages={61\ndash 101},
}

\bib{Br19Lorentz}{article}{
      author={Braverman, M.},
       title={An index of strongly {C}allias operators on {L}orentzian
  manifolds with non-compact boundary},
        date={2019},
        ISSN={1432-1823},
     journal={Mathematische Zeitschrift. Prepublished online
  \texttt{https://link.springer.com/article/10.1007/s00209-019-02270-4}},
         url={https://doi.org/10.1007/s00209-019-02270-4},
}

\bib{BrCecchini17}{article}{
      author={Braverman, M.},
      author={Cecchini, S.},
       title={Callias-type operators in von neumann algebras},
        date={2017},
        ISSN={1559-002X},
     journal={The Journal of Geometric Analysis},
       pages={1\ndash 41},
         url={http://dx.doi.org/10.1007/s12220-017-9832-1},
}

\bib{BrMaschler19}{article}{
      author={Braverman, M.},
      author={Maschler, G.},
       title={Equivariant {A}{P}{S} index for {D}irac operators of non-product
  type near the boundary},
        date={2019},
        ISSN={0022-2518},
     journal={Indiana Univ. Math. J.},
      volume={68},
       pages={435\ndash 501},
}

\bib{BrMiSh02}{article}{
      author={Braverman, M.},
      author={Milatovich, O.},
      author={Shubin, M.},
       title={Essential selfadjointness of {S}chr\"odinger-type operators on
  manifolds},
        date={2002},
     journal={Russian Math. Surveys},
      volume={57},
       pages={41\ndash 692},
}

\bib{BrShi16}{article}{
      author={Braverman, M.},
      author={Shi, P.},
       title={Cobordism invariance of the index of {C}allias-type operators},
        date={2016},
        ISSN={0360-5302},
     journal={Comm. Partial Differential Equations},
      volume={41},
      number={8},
       pages={1183\ndash 1203},
         url={http://dx.doi.org/10.1080/03605302.2016.1183214},
}

\bib{BrShi17b}{article}{
      author={Braverman, M.},
      author={Shi, P.},
       title={{A}{P}{S} index theorem for even-dimensional manifolds with
  non-compact boundary},
     journal={to appear in {\em Communications in Analysis
  and Geometry}},
      eprint={1708.08336},
         url={https://arxiv.org/abs/1708.08336},
}

\bib{BrShi19local}{article}{
      author={Braverman, M.},
      author={Shi, P.},
       title={The index of a local boundary value problem for strongly
  {C}allias-type operators},
        date={2019May},
        ISSN={2199-6806},
     journal={Arnold Mathematical Journal},
      volume={5},
      number={1},
       pages={79\ndash 96},
         url={https://doi.org/10.1007/s40598-019-00110-1},
}

\bib{BruningMoscovici}{article}{
      author={Br{\"u}ning, J.},
      author={Moscovici, H.},
       title={{$L^2$}-index for certain {D}irac-{S}chr\"odinger operators},
        date={1992},
        ISSN={0012-7094},
     journal={Duke Math. J.},
      volume={66},
      number={2},
       pages={311\ndash 336},
         url={http://dx.doi.org/10.1215/S0012-7094-92-06609-9},
      review={\MR{1162192 (93g:58142)}},
}

\bib{Bunke92}{article}{
      author={Bunke, U.},
       title={Relative index theory},
        date={1992},
        ISSN={0022-1236},
     journal={J. Funct. Anal.},
      volume={105},
      number={1},
       pages={63\ndash 76},
         url={http://dx.doi.org/10.1016/0022-1236(92)90072-Q},
}

\bib{Bunke95}{article}{
      author={Bunke, U.},
       title={A {$K$}-theoretic relative index theorem and {C}allias-type
  {D}irac operators},
        date={1995},
        ISSN={0025-5831},
     journal={Math. Ann.},
      volume={303},
      number={2},
       pages={241\ndash 279},
         url={http://dx.doi.org/10.1007/BF01460989},
      review={\MR{1348799 (96e:58148)}},
}

\bib{Callias78}{article}{
      author={Callias, C.},
       title={Axial anomalies and index theorems on open spaces},
        date={1978},
        ISSN={0010-3616},
     journal={Comm. Math. Phys.},
      volume={62},
      number={3},
       pages={213\ndash 235},
         url={http://projecteuclid.org/euclid.cmp/1103904395},
}

\bib{CarvalhoNistor14}{article}{
      author={Carvalho, C.},
      author={Nistor, V.},
       title={An index formula for perturbed {D}irac operators on {L}ie
  manifolds},
    language={English},
        date={2014},
        ISSN={1050-6926},
     journal={The Journal of Geometric Analysis},
      volume={24},
      number={4},
       pages={1808\ndash 1843},
         url={http://dx.doi.org/10.1007/s12220-013-9396-7},
}

\bib{DaiZhang98}{article}{
      author={Dai, X.},
      author={Zhang, W.},
       title={Higher spectral flow},
        date={1998},
        ISSN={0022-1236},
     journal={J. Funct. Anal.},
      volume={157},
      number={2},
       pages={432\ndash 469},
         url={http://dx.doi.org/10.1006/jfan.1998.3273},
      review={\MR{1638328}},
}

\bib{FoxHaskell03}{article}{
      author={Fox, J.},
      author={Haskell, P.},
       title={Heat kernels for perturbed {D}irac operators on even-dimensional
  manifolds with bounded geometry},
        date={2003},
        ISSN={0129-167X},
     journal={Internat. J. Math.},
      volume={14},
      number={1},
       pages={69\ndash 104},
         url={http://dx.doi.org/10.1142/S0129167X03001648},
      review={\MR{1955511}},
}

\bib{FoxHaskell05}{article}{
      author={Fox, J.},
      author={Haskell, P.},
       title={The {A}tiyah-{P}atodi-{S}inger theorem for perturbed {D}irac
  operators on even-dimensional manifolds with bounded geometry},
        date={2005},
        ISSN={1076-9803},
     journal={New York J. Math.},
      volume={11},
       pages={303\ndash 332},
         url={http://nyjm.albany.edu:8000/j/2005/11_303.html},
      review={\MR{2154358}},
}

\bib{Freed98}{article}{
      author={Freed, D.},
       title={Two index theorems in odd dimensions},
        date={1998},
        ISSN={1019-8385},
     journal={Comm. Anal. Geom.},
      volume={6},
      number={2},
       pages={317\ndash 329},
         url={https://doi-org.ezproxy.neu.edu/10.4310/CAG.1998.v6.n2.a4},
      review={\MR{1651419}},
}

\bib{GromovLawson83}{article}{
      author={Gromov, M.},
      author={Lawson, H.~B., Jr.},
       title={Positive scalar curvature and the {D}irac operator on complete
  {R}iemannian manifolds},
        date={1983},
        ISSN={0073-8301},
     journal={Inst. Hautes \'Etudes Sci. Publ. Math.},
      number={58},
       pages={83\ndash 196 (1984)},
         url={http://www.numdam.org/item?id=PMIHES_1983__58__83_0},
}

\bib{Hebey99book}{book}{
      author={Hebey, E.},
       title={Nonlinear analysis on manifolds: {S}obolev spaces and
  inequalities},
      series={Courant Lecture Notes in Mathematics},
   publisher={New York University, Courant Institute of Mathematical Sciences,
  New York; American Mathematical Society, Providence, RI},
        date={1999},
      volume={5},
        ISBN={0-9658703-4-0; 0-8218-2700-6},
      review={\MR{1688256}},
}

\bib{HoravaWitten96}{article}{
      author={Horava, P.},
      author={Witten, E.},
       title={Heterotic and type {I} string dynamics from eleven dimensions},
        date={1996},
        ISSN={0550-3213},
     journal={Nuclear Phys. B},
      volume={460},
      number={3},
       pages={506\ndash 524},
         url={https://doi-org.ezproxy.neu.edu/10.1016/0550-3213(95)00621-4},
      review={\MR{1381609}},
}

\bib{Hormander}{book}{
      author={H\"ormander, L.},
       title={The analysis of linear partial differential operators. {III}},
      series={Classics in Mathematics},
   publisher={Springer, Berlin},
        date={2007},
        ISBN={978-3-540-49937-4},
         url={http://dx.doi.org/10.1007/978-3-540-49938-1},
        note={Pseudo-differential operators, Reprint of the 1994 edition},
      review={\MR{2304165}},
}

\bib{Kato95book}{book}{
      author={Kato, T.},
       title={Perturbation theory for linear operators},
   publisher={Springer},
        date={1995},
  url={http://gen.lib.rus.ec/book/index.php?md5=F433149A2A8FBF8E7274E2E6AB6E537E},
}

\bib{Kottke11}{article}{
      author={Kottke, C.},
       title={An index theorem of {C}allias type for pseudodifferential
  operators},
        date={2011},
        ISSN={1865-2433},
     journal={J. K-Theory},
      volume={8},
      number={3},
       pages={387\ndash 417},
         url={http://dx.doi.org/10.1017/is010011014jkt132},
      review={\MR{2863418}},
}

\bib{Kottke15}{article}{
      author={Kottke, C.},
       title={A {C}allias-type index theorem with degenerate potentials},
        date={2015},
        ISSN={0360-5302},
     journal={Comm. Partial Differential Equations},
      volume={40},
      number={2},
       pages={219\ndash 264},
         url={http://dx.doi.org/10.1080/03605302.2014.942740},
      review={\MR{3277926}},
}

\bib{LawMic89}{book}{
      author={Lawson, H.~B.},
      author={Michelsohn, M.-L.},
       title={Spin geometry},
   publisher={Princeton University Press},
     address={Princeton, New Jersey},
        date={1989},
}

\bib{MelrosePiazza97}{article}{
      author={Melrose, R.~B.},
      author={Piazza, P.},
       title={Families of {D}irac operators, boundaries and the
  {$b$}-calculus},
        date={1997},
        ISSN={0022-040X},
     journal={J. Differential Geom.},
      volume={46},
      number={1},
       pages={99\ndash 180},
         url={http://projecteuclid.org/euclid.jdg/1214459899},
      review={\MR{1472895}},
}

\bib{Muller98}{article}{
      author={M\"uller, W.},
       title={Relative zeta functions, relative determinants and scattering
  theory},
        date={1998},
        ISSN={0010-3616},
     journal={Comm. Math. Phys.},
      volume={192},
      number={2},
       pages={309\ndash 347},
         url={http://dx.doi.org/10.1007/s002200050301},
}

\bib{Nicolaescu95}{article}{
      author={Nicolaescu, L.~I.},
       title={The {M}aslov index, the spectral flow, and decompositions of
  manifolds},
        date={1995},
        ISSN={0012-7094},
     journal={Duke Math. J.},
      volume={80},
      number={2},
       pages={485\ndash 533},
         url={http://dx.doi.org/10.1215/S0012-7094-95-08018-1},
      review={\MR{1369400}},
}

\bib{ReSi1}{book}{
      author={Reed, M.},
      author={Simon, B.},
       title={Methods of modern mathematical physics {I}},
   publisher={Academic Press},
     address={London},
        date={1978},
}

\bib{Shi17}{article}{
      author={Shi, P.},
       title={The index of {C}allias-type operators with {A}tiyah--{P}atodi--{S}inger boundary conditions},
        date={2017},
     journal={Ann. Glob. Anal. Geom.},
      volume={52},
      number={4},
       pages={465\ndash 482},
	     url={http://dx.doi.org/10.1007/s10455-017-9575-z},
      review={\MR{3735908}},
}

\bib{Shi18}{article}{
      author={Shi, P.},
       title={Cauchy data spaces and {A}tiyah-{P}atodi-{S}inger index on
  non-compact manifolds},
        date={201803},
     journal={J. Geom. Phys.},
      volume={133},
       pages={81\ndash 90},
      eprint={1803.01884},
         url={https://arxiv.org/pdf/1803.01884},
}

\bib{Shubin99}{incollection}{
      author={Shubin, M.~A.},
       title={Spectral theory of the {S}chr\"odinger operators on non-compact
  manifolds: qualitative results},
        date={1999},
   booktitle={Spectral theory and geometry ({E}dinburgh, 1998)},
      series={London Math. Soc. Lecture Note Ser.},
      volume={273},
   publisher={Cambridge Univ. Press, Cambridge},
       pages={226\ndash 283},
         url={http://dx.doi.org/10.1017/CBO9780511566165.009},
}

\bib{Wimmer14}{article}{
      author={Wimmer, R.},
       title={An index for confined monopoles},
        date={2014},
        ISSN={0010-3616},
     journal={Comm. Math. Phys.},
      volume={327},
      number={1},
       pages={117\ndash 149},
         url={http://dx.doi.org/10.1007/s00220-014-1934-z},
      review={\MR{3177934}},
}

\end{biblist}
\end{bibdiv}
\end{document}